\documentclass{amsart}%[oneside, 12pt]{amsart}
\usepackage{tikz}
\usepackage{xcolor}
\usepackage{amssymb,latexsym,amsmath,extarrows}
\usepackage{graphicx,mathrsfs,comment}
\usepackage{hyperref,url}

\usepackage{amstext}
\usepackage{bbm}

\numberwithin{equation}{section}

\newtheorem{theorem}{Theorem}[section]
\newtheorem{lemma}[theorem]{Lemma}

\newtheorem{proposition}[theorem]{Proposition}
\newtheorem{remark}[theorem]{Remark}

\newtheorem{definition}[theorem]{Definition}

\newtheorem{conjecture}[theorem]{Conjecture}
\newtheorem{algorithm}{Algorithm}

\def\br{\mathrm{Br}_{\alpha}}

\newcommand{\td}{\tilde}
\newcommand{\al}{\alpha}
\newcommand{\be}{\beta}

\newcommand{\de}{\delta}

\newcommand{\e}{\epsilon}

\newcommand{\om}{\omega}

\newcommand{\co}{\mathcal O}

\newcommand{\wt}{\widetilde}
\newcommand{\wh}{\widehat}

\newcommand{\ZR}{\mathbb{R}}

\newcommand{\bx}{{\bf x}}

\newcommand{\R}{\mathbb{R}}

\newcommand{\Br}{{\rm{Br}}}

\newcommand{\T}{\mathbb{T}  }

\newcommand{\cO}{\mathcal{O}}

\newcommand{\bil}{{\rm{Bil}}}

\newcommand{\rap}{{\rm RapDec}}

\newcommand{\broad}{\mathrm{BL}_{k,A}^p}
\newcommand{\supp}{\textup{supp}}

\begin{document}

\title[Local smoothing estimates]{A note on local smoothing estimates for fractional Schr\"{o}dinger equations}

\author{Shengwen Gan} \address{Shengwen Gan\\  Deparment of Mathematics, Massachusetts Institute of Technology, USA}\email{shengwen@mit.edu}

\author{Changkeun Oh}\address{ Changkeun Oh\\ Department of Mathematics, University of Wisconsin-Madison, USA} \email{coh28@wisc.edu}

\author{Shukun Wu} \address{ Shukun Wu\\  Department of Mathematics\\ California Institute of Technology, USA}\email{skwu@caltech.edu}

\maketitle

\begin{abstract}
We improve local smoothing estimates for fractional Schr\"{o}dinger equations for $\alpha \in (0,1) \cup (1,\infty)$. 
\end{abstract}

%\tableofcontents

%\tableofcontents

\section{Introduction}
Let $n \geq 3$ and $\alpha \in (0,1) \cup (1,\infty)$.
Consider the fractional Schr\"{o}dinger equation
\begin{equation}\label{pde}
         \begin{cases}
            i\frac{\partial u}{\partial t}+(-\Delta)^{\alpha/2}u=0\\
            u(x,0)=g(x). 
                \end{cases}
\end{equation}
The case $\alpha=2$ corresponds to the Schr\"{o}dinger equation.
For $g \in \mathcal{S}(\R^{n-1})$,
the solution of \eqref{pde} can be written as
\begin{equation}
    e^{it(-\Delta)^{\alpha/2}}g(x):=\int_{\R^{n-1}}\widehat{g}(\xi)e(x \cdot \xi+t|\xi|^{\alpha})\,d\xi.
\end{equation}
Here, we use the notation $e(t):=e^{2\pi it}$. Denote  the standard Bessel potential space by $W^{\beta,p}$.  For simplicity, we define $\beta_c=\be_c(n)$ satisfying $\frac{\beta_c}{\alpha}=(n-1)(\frac12-\frac1p)-\frac1p$.
\begin{conjecture}[Local smoothing conjecture]\label{localsmoothing}
Fix a number $\alpha \in (0,1) \cup (1,\infty)$.
For $p>\frac{2n}{n-1}$, it holds that
\begin{equation}\label{12}
    \Big\|e^{it(-\Delta)^{\alpha/2}}g\Big\|_{L^p_{x,t}(\R^{n-1} \times [0,1])} \lesssim \|g\|_{W^{\beta,p}(\R^{n-1})}
 \end{equation}
 for every  $\beta > \beta_c$.
\end{conjecture}

Let $p_{n}$ be the exponent for which the   Fourier restriction problem for paraboloid in $\R^n$ is verified in \cite{hickman2020note}. Note that when $n\geq4$ this is the best exponent for the restriction problem.  Our main theorem is as follows.

\begin{theorem}\label{thm12}
We consider $\alpha$ in the range below
\begin{equation}
         \begin{cases}
            \alpha \in (0,1) \cup (1,\infty) \text{ for } n=3\\
            \alpha \in (1,\infty) \text{ for } n>3.
                \end{cases}
\end{equation}
For $p \geq p_{n}$, it holds that
\begin{equation}\label{ineq1.5}
    \Big\|e^{it(-\Delta)^{\alpha/2}}g\big\|_{L^p_{x,t}(\R^{n-1} \times [0,1])} \lesssim \|g\|_{W^{\beta,p}(\R^{n-1})}
\end{equation}
for every $\beta > \beta_c$.
\end{theorem}

Our strategy is to adapt the method in the Fourier restriction problem to the local smoothing. We remark that the case $\alpha=2$ is already known. Actually, in the work of Rogers \cite{MR2456277}, it is proved that the local smoothing estimate for Schr\"{o}dinger equation follows from the restriction estimate for paraboloid. Thus, when $\al=2$, the recent restriction results in \cite{hickman2020note} and \cite{Wang2018ARE} already imply certain local smoothing estimates.
However, for other $\alpha$, an analogous implication has not been discovered and our result is new. For the case $\alpha>1$, we adapt the ideas of \cite{GOW} and improve the results in
\cite{MR4149067} (Theorem 1.6), \cite{gao2020improved}, and in \cite{GOW} (Corollary 1.5). The case $0<\alpha<1$ is slightly different from the case $\alpha>1$ as the manifold associated to the operator has negative Gaussian curvature. We combine the ideas of \cite{guo2020restriction} and \cite{GOW}, and improve the results in
\cite{MR2629687}.

One of our main tools is the polynomial partitioning, which was introduced by Guth to the study of oscillatory integral operators in \cite{Guth-R3} and \cite{guth2018}. Another main tool is the polynomial Wolff axiom, which was first formulated by Guth and Zahl in \cite{guth2018polynomial}, and later proved by Katz and Rogers in \cite{katz2018polynomial}. Furthermore, a refined version called the nested polynomial Wolff axiom was verified independently in \cite{hickman2019improved} and \cite{zahl2021new}. What's more, 
there is  a stronger result which is a slice version of polynomial Wolff (see \eqref{3destmeasure} and \eqref{estmeasure}). This result was implicitly proved in \cite{zahl2021new}, and we will use it in our proof.

In the appendix, we explore some connections between local smoothing estimates for the fractional Schr\"{o}dinger equations and restriction type estimates. Our argument is based on the ideas in \cite{GOWWZ}, and it generalizes the result in \cite{MR2456277}.

This article is not intended to be self-contained and refers  to \cite{hickman2020note} and \cite{guo2020restriction}.
 
 \medskip
 
\noindent {\bf Organization of the paper.} In Section 2, we reduce the local smoothing estimate to a localized  estimate \eqref{mainest}, and review the wave packet decomposition. In Section 3, we prove Theorem \ref{thm12} for the model case $(\alpha,n)=(2,3)$, and later we will use similar arguments to prove it for other cases. In Section 4 and 5, we prove Theorem \ref{thm12} for the case $\alpha<1$ and $\alpha>1$, respectively. In the appendix, we study the relationship between the local smoothing and Fourier restriction estimates.
 
 \medskip
 
\noindent{\bf Notation.}
 \begin{itemize}
     \item Denote $B^n_r(x)$ the ball of radius $r$ centered at $x$ in $\R^n$. We sometimes write $B_r(x)$ or $B_r$ for $B_r^n(x)$ provided that there will be no confusion. We also sometimes use $B_R^{n}$ and $B^n$ for $B_R^n(0)$ and $B_1^n(0)$, respectively.
     
     \item We write $A(R) \leq \mathrm{RapDec}(R)B$ to mean that for any power $\beta$, there is a constant $C_{N}$ such that
\begin{equation}
\nonumber
    A(R) \leq C_{N}R^{-N}B \;\; \text{for all $R \geq 1$}.
\end{equation}
 \end{itemize}

\noindent {\bf Acknowledgement.}  The authors would like to thank Andreas Seeger for helpful discussions on the case $\alpha<1$. The authors would also like to thank Shaoming Guo, Hong Wang and Ruixiang Zhang for  valuable discussions on the appendix of this manuscript. C.O. was partially supported by the NSF grant DMS-1800274.

\section{Preliminaries}

In this section, we reduce the local smoothing estimate to the localized version of it. We also review the wave packet decomposition.

\subsection{Some reductions}
In this subsection, we do some reductions to make Theorem \ref{thm12} more like a Fourier restriction problem.

We apply some change of variables and assume that the integral of $t$ runs over $[1/2,1]$.
We first consider functions $g$ whose Fourier support is on the annulus $|\xi| \sim 1$. For such $g$, we will prove that
 \begin{equation}\label{220}
     \big\|e^{it(-\Delta)^{\alpha/2}}g\big\|_{L^p_{x,t}(\R^{n-1} \times [R/2,R])} \lesssim R^{(n-1)(\frac12-\frac{1}{p})+\e} \|g\|_{L^{p}(\R^{n-1})}.
\end{equation}
Let us assume this inequality for a moment and prove Theorem \ref{thm12}. By the Littlewood-Paley decomposition, it suffices to show \eqref{ineq1.5} for functions $g_1$ whose Fourier support is contained in the annulus $\{\xi: |\xi| \sim  (R)^{1/\alpha}\}$.
By a simple scaling argument, we have
\begin{equation}
    e^{iRt(-\Delta)^{\alpha/2} }g(x)=R^{-(n-1)/\alpha}(e^{it(-\Delta)^{\alpha/2}}g_1)(R^{-1/\alpha}x)
\end{equation}
where
\begin{equation}
    \widehat{g_1}(\xi)=:\widehat{g}(R^{-1/\alpha}\xi).
\end{equation}
Notice that the function $g$ has Fourier support in the unit ball. By applying \eqref{220} and a scaling argument, we obtain
\begin{equation}
    \|e^{it(-\Delta)^{\alpha/2}}g_1\|_{L^p(\R^{n-1}\times[1/2,1])} \lesssim R^{\e}R^{(n-1)(\frac12-\frac{1}{p})}\|g_1\|_{W^{\beta,p}(\R^{n-1})}
\end{equation}
for every function $g_1$ whose Fourier support is in a ball of radius $R^{1/\alpha}$. 

We have showed that \eqref{ineq1.5} follows from \eqref{220}. We can make a further reduction.
By a standard localization argument (for example, see Lemma 8 of \cite{MR2456277} or Corollary 2.2 of \cite{gao2020improved}), it suffices to prove 
\begin{equation}\label{225}
    \big\|e^{it(-\Delta)^{\alpha/2}}g\big\|_{L^p_{x,t}(B_{R}^{n-1} \times [R/2,R])} \lesssim R^{(n-1)(\frac12-\frac{1}{p})+\e}\|g\|_{L^p(\R^{n-1})}.
 \end{equation}
for every function $g$ whose Fourier support is on $|\xi| \sim 1$.

\smallskip

We will reformulate the inequality \eqref{225} using a different language in order to be consistent with the paper for the Fourier restriction problem.
Let $S$ be a hypersurface in $\R^n$, represented by the graph of a function $h$. Define the extension operator corresponding to the hypersurface $S$ by
\begin{equation}\label{extensionoperator}
Ef(x)=\int_{[-1,1]^{n-1}}f(\xi)
e\big(\xi \cdot x'+h(\xi)x_n \big)\,d\xi,
\end{equation}
where $x=(x',x_n) \in \R^{n-1} \times \R$. Notice that the hypersurface $\{(\xi,|\xi|^{\alpha}) \}$ has all positive principal curvatures for the case $\alpha>1$, and it has principal curvatures with different signs for the case $\alpha<1$. Therefore, the inequality \eqref{225} follows from the following theorem.

\begin{theorem}\label{mainthm}
Suppose that $h: \R^{n-1} \rightarrow \R$ satisfies
\begin{equation}
         \begin{cases}
            \nabla_{\xi\xi}h \text{ has nonzero eigenvalues, for } n=3\\
            \nabla_{\xi\xi}h \text{ has only positive eigenvalues,  for } n>3.
                \end{cases}
\end{equation}
For $p \geq p_{n}$, it holds that 
\begin{equation}\label{mainest}
    \|Ef\|_{L^p(B_R^n)} \leq C_{p,\e} R^{\e}R^{(n-1)(\frac12-\frac{1}{p})}\|\widehat{f}\|_{L^p(\R^{n-1})}
\end{equation}
for every number $R \geq 1$ and  function $f$ supported in $[-1,1]^{n-1}$.
\end{theorem}

\subsection{Wave packet decomposition}

In this subsection, we review the wave packet decomposition.

Let $1 \leq r \leq R$. We consider a hypersurface $\{\xi,h(\xi) \}$ and a function $f$ on $B^{n-1}$. For any $n$-dimensional ball $B_r$, let us do wave packet decomposition of $f$ associated to this ball.

Take a collection $\Theta_r$ of finitely overlapping balls $\theta$ of radius $r^{-1/2}$, and take $\psi_{\theta}$ a smooth partition of unity adapted to this cover, and write $f=\sum_{\theta \in \Theta_r}\psi_{\theta}f$. We next cover $\R^{n-1}$ by finitely overlapping balls of radius $\sim r^{(1+\delta)/2}$ centered at $v \in r^{(1+\delta)/2}\mathbb{Z}^{n-1}$. For each $v$, let $\eta_v$ be a smooth partition of unity adapted to this cover. We now have the decomposition
\begin{equation}
    f=\sum_{(\theta,v) \in \Theta_r \times r^{(1+\delta)/2}\mathbb{Z}^{n-1} }\eta_v^{\vee}*(\psi_{\theta}f).
\end{equation}
We want to make the summand compactly supported, so multiply a function $\tilde{\psi}_{\theta}$, which is supported on $2\theta$ and equal to one on $  r^{-1/2}$-neighborhood of the support of $\psi_{\theta}$. Since $\eta_v^{\vee}$ decays rapidly outside of a ball of radius $\sim r^{-(1+\delta)/2}$ centered at the origin, we have
\begin{equation}
    f=\sum_{(\theta,v) \in \Theta_r \times r^{(1+\delta)/2}\mathbb{Z}^{n-1} }\tilde{\psi}_{\theta}(\eta_v^{\vee}*(\psi_{\theta}f))+\mathrm{RapDec}(r)\|f\|_2.
\end{equation}
Define $f_{\theta,v}$ as the summand on the right hand side of \eqref{recallwpd}.
Denote by $w_{\theta}$  the center of the ball $\theta$.
Let $(\theta,v)$ be given, and  define the tube $T_{\theta,v}$ by
\begin{equation}
    T_{\theta,v}:=\big\{(x',x_n) \in B_r : |x'+x_n\partial_{\om}h(\om_\theta)+v| \leq r^{1/2+\delta} \big\}.
\end{equation}
Define $\T[B_r]$ to be the collection of the tubes $T_{\theta,v}$. 
If the set $T_{\theta,v}$ is empty, then we do not include this set in $\T[B_r]$. We sometimes use the notation 
\begin{equation}\label{recallwpd}
    f_{T_{\theta,v}}:=f_{\theta,v}=\tilde{\psi}_{\theta}(\eta_v^{\vee}*(\psi_{\theta}f)).
\end{equation}  
This finishes the wave packet decomposition on the ball $B_r$.

We next consider the ball $B_r(x_0)$ of radius $r$ centered at the point $x_0$. Define $T_{\theta,v}(x_0):=T_{\theta,v}+x_0$ and the collection of the tubes
\begin{equation}
    \T[B_r(x_0)]:=\{ T_{\theta,v}(x_0): (\theta,v) \in \Theta_r \times r^{(1+\delta)/2}\mathbb{Z}^{n-1} \}.
\end{equation}
Define the phase function $\phi(x,\om):=x' \cdot \om +x_nh(\om)$, and the wave packet
\begin{equation}
    f_{T_{\theta,v}(x_0)}(\om):=e(-\phi(x_0,\om))\Big(f(\cdot)
    e(\phi(x_0,\cdot))\Big)_{\theta,v}(\om).
\end{equation}
Notice that we have the decomposition
\begin{equation}\label{214}
    f=\sum_{(\theta,v)}f_{T_{\theta,v}(x_0)} +\mathrm{RapDec}(r)\|f\|_2.
\end{equation}
A key property of the wave packet is as follows.  We refer to Proposition 2.6 of \cite{Guth-R3} for the details.
\begin{proposition}
For every $x \in B_r(x_0) \setminus T_{\theta,v}(x_0)$, we have
\begin{equation}
    Ef_{T_{\theta,v}(\bx_0)}(x) =\mathrm{RapDec}(r)\|f\|_2.
\end{equation}
\end{proposition}

Therefore, on the ball $B_r(x_0)$, we have the wave packet decomposition
\begin{equation}
    Ef(x)=\sum_{T \in \mathbb{T}[B_r(x_0)] }Ef_{T}(x)+\mathrm{RapDec}(r)\|f\|_2,
\end{equation}
and the function $Ef_{T}$ decays rapidly outside of the tube $T$.

Next we discuss how to compare wave packets at different scales. Suppose that we have two scales $1 \leq \rho \leq r$. Consider balls $B_\rho(x_1) \subset B_r(x_0)$.  Given $\mathbb{W} \subset \T[B_{\rho}(x_1)]$,
define $\uparrow\! \mathbb{W}$ to be a collection of the tubes in $\T[B_r(x_0)]$ satisfying that for every element $T_{\theta,v} \in \uparrow\! \mathbb{W}$ there exists $T_{\theta',v'} \in \mathbb{W}$ such that
\begin{equation}\label{0813.39}
     \mathrm{dist}(\theta, \theta') \lesssim \rho^{-1/2} \text{  \; and \;  } \mathrm{dist}(T_{\theta,v}(x_0)\cap B_\rho(x_1),T_{\theta',v'}(x_1))\lesssim r^{1/2+\delta}.
\end{equation}
For every $\mathbb{W} \subset \T[B_\rho(x_1)]$, we define
\begin{equation}
    g|_{\mathbb{W}}:= \sum_{T \in \mathbb{W} }g_T, \;\;\;\;\; g|_{\uparrow \mathbb{W}}:= \sum_{T \in \uparrow \mathbb{W} }g_T .
\end{equation}
The lemma below relates small wave packets to larger wave packets.

\begin{lemma}[Lemma 5.2 of \cite{hickman2020note}]\label{comparing}
    Given $g \in L^1(B^{n-1})$ and ${\mathbb{W}} \subset \mathbb{T}[B_\rho]$,
    \begin{equation}
        g|_{{\mathbb{W}}}=\big( g|_{\uparrow {\mathbb{W}} } \big)\big|_{\mathbb{W}} + \mathrm{RapDec}(r)\|g\|_2.
    \end{equation}
\end{lemma}

\section{The discussion of a model case: \texorpdfstring{$\alpha=2$}{Lg} in \texorpdfstring{$\mathbb{R}^3$}{Lg}}\label{sec3}

In this section, we prove Theorem \ref{mainthm} for the case $h(\xi)=|\xi|^2$ and $n=3$.
The reason that we discuss this model case is to
avoid irrelevant technical difficulties and
shed light on our method. After we finish the proof of this model case, we will see in later sections that our method can be easily adapted to the proofs for other cases in Theorem \ref{mainthm}. Actually, we classify them into two cases: $0<\al<1$ in $\R^3$; $\al>1$ in any dimensions. Both cases have been studied in the setting of Fourier restriction problem. In later sections, we will derive the local smoothing estimates by combining the techniques in Fourier restriction problem and the method for the model case discussed in this section.

In this section, we assume our Fourier extension operator is 
$$ Ef(x):=\int_{[-1,1]^2} f(\xi) e(\xi\cdot x'+|\xi|^2 x_3 )d\xi, $$
where $x=(x',x_3)\in\R^3$. Our goal is to prove the following estimate.

\begin{theorem}\label{thmsimple}
For $p\ge 13/4$, we have
\begin{equation}\label{simpleest}
    \|Ef\|_{L^p(B_R)} \le C_{p,\e} R^{\e}R^{(1-\frac{2}{p})}\|\widehat{f}\|_{L^p(\R^{2})}
\end{equation}
for every number $R \geq 1$ and  function $f$ supported in $[-1,1]^{2}$.
\end{theorem}

\subsection{Broad function}\label{defbroad}
We follow \cite{Guth-R3} to define the broad function.

Let $[-1,1]^2=\cup \tau$ be a finitely overlapping covering, where $\tau$ are $K^{-1}$-squares. We will later set $K=e^{\e^{-10}}$.
We consider a decomposition $f=\sum_\tau f_\tau$, where $\supp f_\tau\subset \tau$.

For $\alpha\in(0,1)$, we say that $x$ is $\alpha$-broad for $Ef$ if:
\[ \max_\tau|Ef_\tau(x)|\le \al|Ef(x)|. \]
We define $\Br_\al Ef(x)$ to be $Ef(x)$ if $x$ is $\al$-broad for $Ef$ and zero otherwise.

We will prove the following estimate for broad functions. 

\begin{theorem}\label{broadthm}
For any $\e>0$, there exists $K=K(\e)$ such that for every $R\ge 1$, we have
\begin{equation}\label{ineqbroad}
\begin{split}
	\|\mathrm{Br}_{K^{-\e}}Ef\|_{L^{13/4}(B_R)}
	\lesssim_\e R^{\epsilon}R^{2(\frac12-\frac{4}{13})}\|f\|_{L^2}^{8/13}
	\|\widehat{f}\|_{L^{\infty}}^{5/13}.
	\end{split}
	\end{equation}
Moreover, $\lim_{\e\rightarrow0}K(\e)=+\infty$.
\end{theorem}

In the rest of this subsection, we quickly see how to prove \eqref{thmsimple} assuming Theorem \ref{broadthm}. By interpolation, it suffices to prove
\begin{equation}\label{broad1}
\begin{split}
	\|Ef\|_{L^{13/4}(B_R)}
	\leq C_{\e}R^{\epsilon}R^{2(\frac12-\frac{4}{13})}\|f\|_{L^2}^{8/13}
	\|\widehat{f}\|_{L^{\infty}}^{5/13}.
	\end{split}
\end{equation}

We induct on $R$. Assume \eqref{broad1} is true for radius $\le R/2$. We first show the following rescaling lemma.

\begin{lemma}\label{rescale}
Assuming \eqref{broad1} is true for radius $\le R/2$, then for any $K^{-1}$-square $\tau\subset [-1,1]^2$, we have
\begin{equation}\label{rescale1}
\begin{split}
	\|Ef_\tau\|_{L^{13/4}(B_{R})}
	\lesssim C_{\epsilon}K^{-2+6\cdot\frac{4}{13}}R^{\epsilon}R^{2(\frac12-\frac{4}{13})}\|f_\tau\|_{L^2}^{8/13}
	\|\widehat{f}\|_{L^{\infty}}^{5/13}.
	\end{split}
\end{equation}

\end{lemma}
\begin{proof}
Since this parabolic rescaling argument is well-known, we only do the calculation for $\tau$ centered at the origin.

Let $g(\xi)=f_\tau(\xi/K)$. We have
\begin{align*}
    Eg(x)&=\int_{[-K^{-1},K^{-1}]^2}f_\tau(\xi/K)e(\xi\cdot x'+|\xi|^2x_3)d\xi\\
    &=K^2 \int_{[-1,1]^2}f(\xi)e(\xi\cdot K x'+|\xi|^2 K^2 x_3)d\xi\\
    &=K^2 Ef_\tau(Kx',K^2x_3).
\end{align*} 
By change of variable, we have $$ \|Ef_\tau\|_{L^{13/4}(B_{R})}\le K^{-14/13}\|Eg\|_{L^{13/4}([-R/K,R/K]^2\times [-R/K^2,R/K^2])}. $$
Also, we have $ \|f_\tau\|_{L^2}= K^{-1}\|g\|_{L^2}$, $\|\widehat f_\tau\|_{L^\infty}=K^{-2} \|\widehat g\|_{L^\infty}$. Since $f_\tau=\chi_\tau f$ such that $\widehat \chi_\tau$ is $L^1$-bounded, we have $\|\widehat g\|_{L^\infty}=K^2\|\widehat f_\tau\|_{L^\infty}\lesssim K^{2}\|\widehat f\|_{L^\infty} $.

We divide $[-R/K,R/K]^2\times [-R/K^2,R/K^2]$ into $\sim K$ balls of radius $R/K^2$. 
Since $g$ is supported in $[-1,1]^2$. For each such ball $B_{R/K^2}$, by the hypothesis we have
$$ \|Eg\|_{L^{13/4}(B_{R/K^2})}
	\leq C_{\e}(R/K^2)^{\epsilon}(R/K^2)^{2(\frac12-\frac{4}{13})}\|g\|_{L^2}^{8/13}
	\|\widehat{g}\|_{L^{\infty}}^{5/13}. $$
Summing over $B_{R/K^2}$, we get
$$ \|Eg\|_{L^{13/4}([-R/K,R/K]^2\times [-R/K^2,R/K^2])}
	\leq C_{\e}K^{\frac{4}{13}}(R/K^2)^{\epsilon}(R/K^2)^{2(\frac12-\frac{4}{13})}\|g\|_{L^2}^{8/13}
	\|\widehat{g}\|_{L^{\infty}}^{5/13}. $$
Plugging in $f_\tau$ and noting that $\|\widehat f_\tau\|_{L^\infty}\lesssim\|\widehat f\|_{L^\infty} $, we obtain
$$ \|Ef_\tau\|_{L^{13/4}(B_{R})}
	\lesssim C_{\e}K^{-2+6\cdot\frac{4}{13}}R^{\epsilon}R^{2(\frac12-\frac{4}{13})}\|f_\tau\|_{L^2}^{8/13}
	\|\widehat{f}\|_{L^{\infty}}^{5/13}, $$
which finishes the proof.
\end{proof}

We continue the proof of \eqref{broad1}. Recall $K=e^{\e^{-10}}$ and we set $\al=K^{-\e}$. By the definition of the broad function, we have
\begin{equation}
    |Ef(x)|\le |\Br_\al Ef(x)|+\al^{-1}\sup_\tau|Ef_\tau(x)|.
\end{equation}
As a result, we obtain
\begin{equation}
    \int_{B_R}|Ef(x)|^{13/4}\le  \int_{B_R}|\Br_\al Ef(x)|^{13/4}+\al^{-1}\sum_\tau \int_{B_R}|Ef_\tau(x)|^{13/4}.
\end{equation}
Using Theorem \ref{broadthm}, we can bound the first term by $$\frac{1}{10}\big(C_{\e}R^{\epsilon}R^{2(\frac12-\frac{4}{13})}\big)^{13/4}\|f\|_{L^2}^{2}
\|\widehat{f}\|_{L^{\infty}}^{1.25},$$ 
if $C_{\e}$ is large enough. For the second term, we apply Lemma \ref{rescale} to obtain
\begin{equation}
    \begin{split}
        \al^{-1}\sum_\tau \int_{B_R}|Ef_\tau(x)|^{13/4}&\lesssim K^\e \sum_\tau \big(C_{\e}K^{-2+6\cdot\frac{4}{13}}R^{\epsilon}R^{2(\frac12-\frac{4}{13})}\big)^{13/4}\|f_\tau\|_{L^2}^{2}
	\|\widehat{f}\|_{L^{\infty}}^{1.25}\\
	&\le \frac{1}{10}\big(C_{\e}R^{\epsilon}R^{2(\frac12-\frac{4}{13})}\big)^{13/4}\|f\|_{L^2}^{2}
	\|\widehat{f}\|_{L^{\infty}}^{1.25}.
    \end{split}
\end{equation}
Here we use $\sum_\tau\|f_\tau\|_2^2\lesssim \|f\|_2^2$ and we assume $K$ is large enough so that the negative power of $K$ cancels the implicit constant.

Combining things together, we proved \eqref{broad1}.

\smallskip

In the remaining of the section, we prove
Theorem \ref{broadthm}.
We choose $K(\e)=e^{\e^{-10}}$. Without loss of generality, we can assume the wave packets of $f$ are contained in $B_{10R}$ since those outside $B_{10R}$ contribute little to the left hand side of \eqref{ineqbroad}. The proof includes several steps. First, we will do a one-step polynomial partitioning, where we will deal with three cases: cellular case, transverse case and tangent case. Next, we iterate the polynomial partitioning until we encounter the tangent case or the radius becomes small. Finally, we will combine the estimates from polynomial partitioning and the sliced polynomial Wolff estimate to derive our result.

\subsection{One-step polynomial partitioning}\label{sec32}
First, we recall the polynomial partitioning lemma. Here we use the modified version which gives additional information on the radius of cells. For more details, we refer to \cite{Wang2018ARE} page 10.

\begin{lemma} \label{partition}
 Let $F$ be a non-negative $L^1$ function on $\R^n$. Then for any $D \in \mathbb{Z}^{+}$, there is a non-zero polynomial $P$ of degree at most $D$ so that $\mathbb{R}^n \setminus Z(P)$ is a disjoint union of $\sim D^n$ open sets $O_i$, and the integrals $\int_{O_i} F$ agree up to a factor of 2. Moreover, the polynomial $P$ is a product of non-singular polynomials, and each set $O_i$ is contained in a ball of radius $R/D$. 
\end{lemma}

Take the parameters $
D=R^{\e^6}$ and $ \delta=\e^2$.

Let us apply Lemma \ref{partition} to the function  $F=|1_{B_R}\br Ef|^{13/4}$. Then there exists a non-zero polynomial $P$ of degree at most $D$ so that $\R^3\setminus Z(P)$ is a disjoint union of $\sim D^3$ cells $O_i$ and so that for each $i$,
\begin{equation}\label{poly0}
    \int_{O_{i} \cap B_R }(\br{Ef})^{13/4}
    \sim D^{-3}
    \int_{B_R }(\br{Ef})^{13/4}.
\end{equation}
Moreover, we can assume that $P$ is a product of non-singular polynomials.

Define the wall and the cells by
\begin{equation}
    W:=N_{R^{1/2+\delta}}(Z(P)), \;\;\; O_i^{(1)}:=(O_i \cap B_R) \setminus W.
\end{equation}
Here, the superscript ``$(1)$" on $O_i^{(1)}$ indicates that we are doing the first polynomial partition.
We have the following estimate
\begin{equation}\label{cellwall}
    \|\br Ef\|_{L^{13/4}(B_R)}^{13/4} \lesssim \sum_{O_i^{(1)}  }\|\mathrm{Br}_{\alpha}Ef\|_{L^{13/4}(O_i^{(1)}) }^{13/4}
    +\|\mathrm{Br}_{\alpha}Ef\|_{L^{13/4}(W) }^{13/4}.
\end{equation}

We say that we are in the cellular case if the first term dominates the second term. Otherwise, we say that we are in the wall case.
\\

Let us deal with the integral of the broad function on $W$. We cover $W$ by balls $B_k$ of radius $R^{1-\delta}$. For each $B_k$, we will define tangent tubes and transverse tubes associated to $B_k$.

\begin{definition}[Tangent tubes]\label{deftang}
Define $\T_{k,-}$ to be the set of all tubes $T \in \T[B_R]$ obeying the following conditions:

$\bullet$
$T\cap B_k \cap W \neq \emptyset$.

$\bullet$ If $z$ is any non-singular point of $Z(P)$ lying in $2B_k\cap 10T$, then
\begin{equation}
\mathrm{Angle}(v(T),{T}_z(Z(P))) \leq R^{-1/2+2\delta}.
\end{equation}
Recall that $v(T)$ is the unit vector in the direction of the tube $T$.
\end{definition}

\begin{definition}[Transverse tubes]\label{deftrans}

Define $\T_{k,+}$ to be the set of all tubes $T \in \T[B_R]$ obeying the following conditions:

$\bullet$
$T\cap B_k \cap W \neq \emptyset$.

$\bullet$ There exists a non-singular point $z$ of $Z(P)$ lying in $2B_k\cap 10T$, so that
\begin{equation}
\mathrm{Angle}(v(T),{T}_z(Z(P))) > R^{-1/2+2\delta}.
\end{equation}
\end{definition}

Recall that $[-1,1]^2$ is divided into $\sim K^2$ squares $\tau$ of diameter $K^{-1}$. 
We apply the wave packet decomposition \eqref{214} to $f_{\tau}$, so that
\begin{equation}
 f_{\tau}=\sum_T f_{\tau,T}+\mathrm{RapDec(R)}\|f\|_2.   
\end{equation}
We define
\begin{equation}\label{313}
     f_{\tau,k,+}:=\sum_{T \in \mathbb{T}_{k,+} }f_{\tau,T}, \;\; f_{k,+}:=\sum_{\tau}f_{\tau,k,+}.
\end{equation} 
For $I$ being a collection of these $\tau$'s, we define
\begin{equation}\label{314}
    f_{I,k,+}:=\sum_{\tau \in I}f_{\tau,k,+}.
\end{equation} 
Define $f_{k,-}$ similarly. Lastly, we define the bilinear function
\begin{equation}\label{defbil}
\mathrm{Bil}(Ef_{k,-}):=\sum_{\tau,\tau'\ \textup{non adjacent}} |Ef_{\tau,k,-}|^{1/2}|Ef_{\tau',k,-}|^{1/2}.
\end{equation}
We have the following lemma:
\begin{lemma}[Lemma 3.8 of \cite{Guth-R3}]\label{broadnarrow}
If $x\in B_k\cap W$ and $\al=K^{-\e}$, then 
\begin{equation}\label{transtang}
    \Br_\al|Ef(x)|\le 2\bigg( \sum_I\Br_{2\al}Ef_{I,k,+}(x)+K^{100}\bil(Ef_{k,-})(x)+\rap(R)\|f\|_2 \bigg).
\end{equation}
\end{lemma}
Here the summation $\sum_I$ is summing over the roughly $\sim 2^{K^2}$ subsets of the set of caps $\tau$.

Now we plug \eqref{transtang} into \eqref{cellwall} to obtain:
\begin{equation}\label{3term}
\begin{split}
    \|\Br_\al Ef\|_{L^{13/4}(B_R)}^{13/4}&\lesssim \sum_{O_i^{(1)}}\|\Br_\al Ef\|^{13/4}_{L^{13/4}(O_i^{(1)})}+\sum_k\sum_I\|\Br_{2\al}Ef_{I,k,+}\|^{13/4}_{L^{13/4}(B_k\cap W)}\\
    &+K^{100}\|\bil(Ef_{j,-})\|_{L^{13/4}(W)}^{13/4}+\rap(R)\|f\|_2
\end{split}
\end{equation}
There are three cases: If the first term on the right hand side dominates, we say we are in the \textit{cellular case}; If the second term on the right hand side dominates, we say we are in the \textit{transverse case}; If the third term on the right hand side dominates, we say we are in the \textit{tangent case}.

\subsubsection{Cellular case}
Suppose that we are in the cellular case. Denote by $\mathcal{O}^{(1)}$ the collection of the cells $O_i^{(1)}$, and by $\T_{i}$ the collection of the tubes in $\T[B_R]$ intersecting the cell $O_i^{(1)}$. Notice that the cardinality of $\mathcal{O}^{(1)}$ is comparable to $D^3$. Define
\begin{equation}
    f_{i}:=\sum_{\tau}\sum_{T \in \T_i}f_{\tau,T}.
\end{equation}
By Lemma 3.7 of \cite{Guth-R3}, for every $x \in O_i^{(1)}$, we have
\begin{equation}
\mathrm{Br}_{\alpha}Ef(x)
\leq 2\mathrm{Br}_{2\alpha}Ef_i(x) + \rap(R)\|f\|_{L^2}.
\end{equation}
Since we are in the cellular case, we have
\begin{equation}\label{cell1}
    \|\br Ef\|_{L^{13/4}(B_R)}^{13/4} \lesssim \sum_{O_i^{(1)} }\|\mathrm{Br}_{\alpha}Ef\|_{L^{13/4}(O_i^{(1)}) }^{13/4}
    \lesssim \sum_{O_i^{(1)} }\|\mathrm{Br}_{2\alpha}Ef_i\|_{L^{13/4}(O_i^{(1)}) }^{13/4}.
\end{equation}
Notice that by Lemma 3.2 of \cite{Guth-R3}, each tube $T \in \T$ can intersect at most $D+1$ many cells in $\co^{(1)}$, thus
\begin{equation}\label{cell2}
    \sum_{i }\|f_{i}\|_2^2 \lesssim D \|f\|_2^2.
\end{equation}
Since the cardinality of the cells $\mathcal{O}^{(1)}$ is comparable to $D^3$, by pigeonholing, we can take a sub-collection of the cells so that the cardinality of which is comparable to $D^3$ and for each $O_i^{(1)}$ in this sub-collection we have
\begin{equation}\label{cell3}
    \|f_{i}\|_2^2 \lesssim D^{-2}\|f\|_2^2.
\end{equation}
By abusing the notation, we still denote this sub-collection by $\mathcal{O}^{(1)}$. 
By \eqref{cell1} and \eqref{poly0}, we  have
\begin{equation}\label{cell4}
    \|\br Ef\|_{L^{13/4}(B_R)}^{13/4} 
    \lesssim \sum_{ O_i^{(1)} \in \mathcal{O}^{(1)} }\|\mathrm{Br}_{2\alpha}Ef_i\|_{L^{13/4}(O_i^{(1)}) }^{13/4}.
\end{equation}
To indicate the relation between functions and cells, we write $f_{O_i^{(1)}}:=f_i$ for any $O_i^{(1)}\in\co^{(1)}$. For convenience, later we will drop the subscript $i$ to simply write $f_{O^{(1)}}$ for any $O^{(1)}\in\co^{(1)}$.

\subsubsection{Transverse case}
Suppose that we are in the transverse case. We have
\begin{equation*}
\begin{split}
\int_{B_R \cap W}
|\mathrm{Br}_{\alpha}Ef(x)|^{13/4}
\lesssim \sum_{k}\sum_I\int_{B_k \cap W} \Big( \mathrm{Br}_{2\alpha }Ef_{I,k,+} \Big)^{13/4}.
\end{split}
\end{equation*}
Choose the $I$ that maximizes the right hand side, so we have
\begin{equation}\label{trans1}
\begin{split}
\int_{B_R \cap W}
|\mathrm{Br}_{\alpha}Ef(x)|^{13/4}
\lesssim 2^{K^2} \sum_{k}\int_{B_k \cap W} \Big( \mathrm{Br}_{2\alpha }Ef_{I,k,+} \Big)^{13/4}.
\end{split}
\end{equation}

We set $\co^{(1)}:=\{B_k\cap W\}$ to be the collection of these cells $B_k\cap W$. For $O^{(1)}=B_k\cap W\in\co^{(1)}$, we write $f_{O^{(1)}}:=f_{I,k,+}$.
Note that $f_{O^{(1)}}$ is the sum of a subset of wave packets of $f$, so
\begin{equation}\label{L2small}
    \|f_{O^{(1)}}\|_2^2\lesssim \|f\|_2^2.
\end{equation}

By Lemma 5.7 of \cite{guth2018}, each tube belongs to at most $O(D^3)$ different sets $\T_{k,+}$, so we have
\begin{equation}\label{trans2}
    \sum_{O^{(1)}\in\co^{(1)}} \|f_{O^{(1)}}\|_2^2 \lesssim D^3\|f\|_2^2.
\end{equation}

\subsubsection{Tangent case}
Suppose we are in the tangent case, then
\begin{equation}\label{tang1}
\begin{split}
\int_{B_R \cap W}
|\mathrm{Br}_{\alpha}Ef(x)|^{13/4}
&\lesssim K^{O(1)}\sum_k \int_{B_k \cap W}\mathrm{Bil}(Ef_{k,-})^{13/4}.
\end{split}
\end{equation}

We need the following estimate for the bilinear term.
\begin{lemma}[(43) of \cite{Guth-R3}]\label{tang2}
\begin{equation}\label{327}
\begin{split}
\|\mathrm{Bil}(Ef_{k,-})\|_{L^{13/4}(B_k \cap W)}
\lesssim R^{O(\delta)}R^{\frac{1}{52}}\big(\sum_{\tau} \|f_{\tau,k,-}\|_2^2 \big)^{1/2}.
\end{split}
\end{equation}
\end{lemma}

Choose the $k$ and $\tau$ that maximizes the right hand side of \eqref{tang1} and apply Lemma \ref{tang2}, so we obtain
\begin{equation}\label{tang3}
\begin{split}
\int_{B_R \cap W}
|\mathrm{Br}_{\alpha}Ef(x)|^{13/4}
&\lesssim R^{O(\delta)}R^{1/16}\|f_{\tau,k,-}\|_2^{13/4}.
\end{split}
\end{equation}
We set $\co^{(1)}$ to consist only a single cell $B_k\cap W$, and for $O^{(1)}=B_k\cap W$ set $f_{O^{(1)}}:=f_{\tau,k,-}$. We also have \begin{equation}\label{L2smalltan}
    \|f_{O^{(1)}}\|_2^2\lesssim \|f\|_2^2,
\end{equation} 
since $f_{O^{(1)}}$ is the sum of a subset of wave packets of $f$.

\subsection{Iteration}\label{sec33}
In the last subsection, we work with the function $\Br_\al Ef$ in the cell $B_R$ and result in one the three cases. We will iterate this process in this subsection. We would like to state this iteration as an algorithm.

\begin{algorithm}[Iteration of polynomial partitioning]\label{algo}\hfill

{\rm
Suppose we begin with a function $f$, a cell $B_R$ and a small number $\al=K^{-\e}$. By iteratively doing the polynomial partitioning, we have the following {\bf output}:

\smallskip

\noindent$\bullet$ There exists an integer $s$ ($1\le s\le \e^{-O(1)}$) which denotes the total number of iteration steps. There is a function STATE which we use to record the state of each step:
\begin{equation}
    \textup{STATE}:\{1,2,\cdots,s\}\rightarrow \{\textup{cell,~trans,~tang}\}.
\end{equation}
We stop at step $s$ because of one of the two reasons: either we first encounter the tangent state STATE($s$)=tang, or the radius at step $s$ is small $r_s\sim R^{\e/10}$ ($r_s$ is as below). 

\smallskip

\noindent$\bullet$ At each step $u$, $u\in\{1,\cdots,s\}$, we have:

\smallskip

\noindent{\bf 1.} A radius $r_u$ which is defined recursively by
\begin{equation}\label{it1.1}
    r_u=
\left\{
\begin{array}{cc}
    r_{u-1}/D,\ \ &\textup{STATE}(u)=\textup{cell}  \\
    r_{u-1}^{1-\de},\ \ &\textup{STATE}(u)=\textup{trans}. 
\end{array}
\right.
\end{equation}
We also set $r_0=R$.
Here are another two important parameters defined by
\begin{align}
    &s_c:=\#\{1\le i\le s:~\textup{STATE}(i)=\textup{cell}\}, \\ 
    &s_t:=\#\{1\le i\le s:~\textup{STATE}(i)=\textup{trans}\}.
\end{align}
Note that we have
\begin{align}
    &s_c\lesssim \log R/\log D=\e^{-6},\\
    &s_t\lesssim \de^{-2}= \e^{-4}.
\end{align}

\noindent{\bf 2.} A set of cells $\co^{(u)}=\{O^{(u)}\}$ such that each $O^{(u)}$ is contained in a $r_u$-ball $B_{O^{(u)}}$.  Each $O^{(u)}$ has a unique parent $O^{(u-1)}\in\cO^{(u-1)}$. We denote this relation by 
\begin{equation}
    O^{(u)}<O^{(u-1)}.
\end{equation}
Moreover we have the nested property for these cells. That is, for any cell $O_{s}\in\co_{s}$, there exist unique $O^{(u)}\in \co^{(u)}$ $(u=1,\cdots,s)$ such that
\begin{equation}
    O_s<O_{s-1}<\cdots<O_1(<B_R). 
\end{equation}

\noindent{\bf 3.} A set of functions $\{f_{O^{(u)}}\}_{O^{(u)}\in\co^{(u)}}$.

\smallskip

\noindent{\bf 4.} There are three possible cases for each step $u$: cellular case, transverse case and tangent case. The outputs for each case are the following:

\smallskip

\noindent\textit {Cellular case}: If $\textup{STATE}(u)=$ cell, we have the following estimates:

\begin{equation}\label{itcellular0}
    \|f_{O^{(u)}}\|_2^2\lesssim D^{-2}\|f_{O^{(u-1)}}\|_2^2,\ \ \textup{for~}O^{(u)}<O^{(u-1)}.
\end{equation}

\begin{equation}\label{itcellular1}
    \sum_{O^{(u)}}\|f_{O^{(u)}}\|_2^2\lesssim D\sum_{O^{(u-1)}}\|f_{O^{(u-1)}}\|_2^2.
\end{equation}

\begin{equation}\label{itcellular2}
    \sum_{O^{(u-1)}}\|\Br_{2^{u-1}\al}Ef_{O^{(u-1)}}\|_{L^{13/4}(O^{(u-1)})}^{13/4}\lesssim \sum_{O^{(u)}} \|\Br_{2^{u}\al}Ef_{O^{(u)}}\|_{L^{13/4}(O^{(u)})}^{13/4}.
\end{equation}

\medskip

\noindent\textit{Transverse case}: If $\textup{STATE}(u)=$ trans, we have the following estimates:

\begin{equation}\label{ittrans0}
    \|f_{O^{(u)}}\|_2^2\lesssim \|f_{O^{(u-1)}}\|_2^2,\ \ \textup{for~}O^{(u)}<O^{(u-1)}.
\end{equation}

\begin{equation}\label{ittrans1}
    \sum_{O^{(u)}}\|f_{O^{(u)}}\|_2^2\lesssim D^3 \sum_{O^{(u-1)}}\|f_{O^{(u-1)}}\|_2^2.
\end{equation}

\begin{equation}\label{ittrans2}
    \sum_{O^{(u-1)}}\|\Br_{2^{u-1}\al}Ef_{O^{(u-1)}}\|_{L^{13/4}(O^{(u-1)})}^{13/4}\lesssim 2^{K^2} \sum_{O^{(u)}} \|\Br_{2^{u}\al}Ef_{O^{(u)}}\|_{L^{13/4}(O^{(u)})}^{13/4}.
\end{equation}
\medskip

\noindent\textit{Tangent case}: If $\textup{STATE}(u)=$ tang, so $u=s$, then we have the following estimates:

\begin{equation}\label{ittang0}
    \|f_{O^{(u)}}\|_2^2\lesssim \|f_{O^{(u-1)}}\|_2^2,\ \ \textup{for~}O^{(u)}<O^{(u-1)}.
\end{equation}

\begin{equation}\label{ittang1}
    \sum_{O^{(u)}}\|f_{O^{(u)}}\|_2^2\lesssim \sum_{O^{(u-1)}}\|f_{O^{(u-1)}}\|_2^2.
\end{equation}

\begin{equation}\label{ittang2}
    \sum_{O^{(u-1)}}\|\Br_{2^{u-1}\al}Ef_{O^{(u-1)}}\|_{L^{13/4}(O^{(u-1)})}^{13/4}\lesssim K^{O(1)} \sum_{O^{(u)}} \|\Br_{2^{u}\al}Ef_{O^{(u)}}\|_{L^{13/4}(O^{(u)})}^{13/4}.
\end{equation}

\begin{equation}\label{ittang3}
    \|\Br_{2^{u}\al}Ef_{O^{(u)}}\|_{L^{13/4}(O^{(u)})}^{13/4}\lesssim r_{s-1}^{1/16+\de} \sum_{O^{(u)}} \|f_{O^{(u)}}\|_2^{13/4}.
\end{equation}
}
\end{algorithm}

\begin{proof}
The proof follows from the same idea as in the last subsection. Note that in the last subsection we have dealt with $u=1$. We briefly explain how to pass from step $u-1$ to step $u$.

Suppose we already obtained cells $\co^{(u-1)}=\{O^{(u-1)}\}$ and functions $$\{f_{O^{(u-1)}}\}_{O^{(u-1)}\in\co^{(u-1)}},$$ and note that each cell $O^{(u-1)}$ is contained in a ball of radius $r_{u-1}$. Now, we fix a cell $O^{(u-1)}\in\co^{(u-1)}$. We apply the same argument as in the last subsection, with $R$ replaced by $r_{u-1}$, $B_R$ replaced by $O^{(u-1)}$, $f$ replaced by $f_{O^{(u-1)}}$.
We just give a sketch here.

As in the Subsection \ref{defbroad}, we still use $K^{-1}$-squares $\{\tau\}$ to cover $[-1,1]^2$ and define the broad function $\Br_\al E f_{O^{(u-1)}}$.
Apply polynomial partitioning to the function $|1_{O^{(u-1)}}\Br_{2^{(u-1)}}Ef_{O^{(u-1)}}|^{13/4}$, then we obtain $\sim D^3$ cells $\co^{(u)}(O^{(u-1)})=\{O_i^{(u)}\}$ and a wall denoted by $W_{O^{(u-1)}}$ such that an inequality similar to \eqref{cellwall} holds:
\begin{equation}\label{itcellwall}
    \|\Br_{2^{u-1}\al} Ef_{O^{(u-1)}}\|_{L^{13/4}(O^{(u-1)})}^{13/4} \lesssim \sum_{O_i^{(u)}  }\|\mathrm{Br}_{2^u\alpha}Ef\|_{L^{13/4}(O_i^{(u)}) }^{13/4}
    +\|\mathrm{Br}_{2^u\alpha}Ef\|_{L^{13/4}(W_{O^{(u-1)}}) }^{13/4}.
\end{equation}
To deal with the integral over wall, we cover $W_{O^{(u-1)}}$ by balls $B_k$ of radius $r_{u-1}^{1-\de}$. For each $B_k$, similar to Definition \ref{deftang} and \ref{deftrans}, we can define the tangent tubes $\T_{k,-}$ and transverse tubes $\T_{k,+}$ associated to $B_k$, but the tubes here are of dimensions $r_{u-1}^{1/2}\times r_{u-1}^{1/2}\times r_{u-1}$. We also remark that we do wave packet decomposition for $f_{O^{(1)},\tau}$ at scale $r_{u-1}$. Similarly we can define the bilinear function $\bil(E f_{O^{(u-1)},k,-})$ as in \eqref{defbil}. 

We also have the following lemma which is similar to Lemma \ref{broadnarrow}:
\begin{lemma}\label{itbroadnarrow}
If $x\in B_k\cap W_{O^{(u-1)}}$ and $\al=K^{-\e}$, then 
\begin{equation}\label{ittranstang}
\begin{split}
    \Br_{2^{u-1}}\al|Ef_{O^{(u-1)}}(x)|\lesssim \sum_I\Br_{2^u\al}Ef_{O^{(u-1)},I,k,+}(x)+K^{100}\bil(Ef_{O^{(u-1)},k,-})(x)\\
    +\rap(R)\|f\|_2.
    \end{split}
\end{equation}
\end{lemma}
Now we plug \eqref{ittranstang} into \eqref{itcellwall} to obtain:
\begin{equation}
\begin{split}
    \|\Br_{2^{u-1}\al} Ef_{O^{(u-1)}}\|_{L^{13/4}(O^{(u-1)})}^{13/4}&\lesssim \sum_{O_i^{(u)}}\|\Br_{2^u\al} Ef_{O^{(u-1)}}\|^{13/4}_{L^{13/4}(O_i^{(u)})}\\
    &+\sum_k\sum_I\|\Br_{2^u\al}Ef_{O^{(u-1)},I,k,+}\|^{13/4}_{L^{13/4}(B_k\cap W_{O^{(u-1)}})}\\
    &+K^{100}\|\bil(Ef_{O^{(u-1)},j,-})\|_{L^{13/4}(W_{O^{(u-1)}})}^{13/4}\\
    &+\rap(R)\|f\|_2.
\end{split}
\end{equation}

Summing over $O^{(u-1)}\in \co^{(u-1)}$, we obtain
\begin{equation}\label{it3term}
\begin{split}
    \sum_{O^{(u-1)}}\|\Br_{2^{u-1}\al} Ef_{O^{(u-1)}}\|_{L^{13/4}(O^{(u-1)})}^{13/4}&\lesssim \sum_{O^{(u-1)}}\sum_{O_i^{(u)}\in\co^u(O^{(u-1)})}\|\Br_{2^u\al} Ef_{O^{(u-1)}}\|^{13/4}_{L^{13/4}(O_i^{(u)})}\\
    &+\sum_{O^{(u-1)}}\sum_k\sum_I\|\Br_{2^u\al}Ef_{O^{(u-1)},I,k,+}\|^{13/4}_{L^{13/4}(B_k\cap W_{O^{(u-1)}})}\\
    &+K^{100}\sum_{O^{(u-1)}}\|\bil(Ef_{O^{(u-1)},j,-})\|_{L^{13/4}(W_{O^{(u-1)}})}^{13/4}\\
    &+\rap(R)\|f\|_2
\end{split}
\end{equation}

There are three cases: If the first term on the right hand side dominates, we set STATE$(u)$=cell; If the second term on the right hand side dominates, we set STATE$(u)$=trans; If the third term on the right hand side dominates, we set STATE$(u)$=tang.

When STATE$(u)$=cell, we can find cells $\cO^{(u)}=\{O^{(u)}\}$ which is a subset of $\bigcup_{O^{(u-1)}\in\co^{(u-1)}}\cO^{(u)}(O^{(u-1)})$, and functions $\{f_{O^{(u)}}\}$, so that similar to \eqref{cell3}, \eqref{cell2} and \eqref{cell4}, we have \eqref{itcellular0}, \eqref{itcellular1} and \eqref{itcellular2}.

When STATE$(u)$=trans, we remark that \eqref{ittrans0}, \eqref{ittrans1} and \eqref{ittrans2} are analogues of \eqref{L2small}, \eqref{trans2} and \eqref{trans1}.

When STATE$(u)$=tang, we remark that \eqref{ittang0} is an analogue of \eqref{L2smalltan}. \eqref{ittang1} holds because each $O^{(u-1)}$ has only one child: $O^{(u)}<O^{(u-1)}$ in the tangent case. \eqref{ittang2} is an analogue of \eqref{tang1}. Also, \eqref{ittang3} is an analogue of \eqref{tang3}

We also remark that the relations ``$<$" between cells in $\co^{(u)}$ and cells in $\co^{(u-1)}$ are apparently defined according to the polynomial partitioning process: we say $O^{(u)}<O^{(u-1)}$ if $O^{(u)}$ is obtained by doing polynomial partitioning to $O^{(u-1)}$.

We finally remark that the subscript in $\Br_{2^u\al}$ is always less than $1$, since  $2^{s}\al\lesssim 2^{O(\e^{-6})}K^{-\e}<1$, so the broad functions are always well-defined here.
\end{proof}

\subsection{Put things together}
As stated in Algorithm \ref{algo}, the iteration stops due to one of the two reasons: STATE($s$)=tang; $r_s\sim R^{\e/10}$.

Let us consider the case that the iteration stops due to the smallness of the radius $r_s$. Iterating \eqref{itcellular2} and \eqref{ittrans2}, we have
\begin{equation}
    \|\br Ef\|_{L^{13/4}(B_R)}^{13/4}\lesssim \sum_{O^{(s)}}  \|\Br_{2^s\al} Ef_{O^{(s)}}\|_{L^{13/4}(O^{(s)})}^{13/4}. 
\end{equation}
 Also note that the diameter of $O^{(s)}$ is $\lesssim r_s\sim R^{\e/10}$, so
\begin{equation}\label{plugin}
\begin{split}
    \|\br Ef\|_{L^{13/4}(B_R)}^{13/4}\lesssim \sum_{O^{(s)}}  \| Ef_{O^{(s)}}\|_{L^{13/4}(O^{(s)})}^{13/4} \lesssim R^{\e/3}\sum_{O^{(s)}} \|f_{O^{(s)}}\|_2^{13/4}\\
    \lesssim 
    R^{\epsilon/3}\sum_{O^{(s)}} \|f_{O^{(s)}}\|_2^{2} \, \big( \max_{O^{(s)}}\|f_{O^{(s)}}\|_2^{5/4} \big).
    \end{split}
\end{equation}
By iterating \eqref{itcellular1} and \eqref{ittrans1} and noting $D^{3s_t}\lesssim R^{O(\e^2)}$, we have
\begin{equation}\label{est1}
    \sum_{O^{(s)}} \|f_{O^{(s)}}\|_2^{2}\lesssim R^{O(\e^2)} D^{s_c}\|f\|_2^2.
\end{equation}
By iterating \eqref{itcellular0} and \eqref{ittrans0}, we have the inequality $\|f_{O^{(s)}}\|_2^{2}\lesssim D^{-2s_c}\|f\|_2^2$ for each $O^{(s)}\in\co^{(s)}$.
Since we assumed the wave packet of $f$ is contained in $B_{10R}$, we have $\wh f$ is essentially supported in $[-10R,10R]^2$. By H\"older's inequality, we have $\|f\|_2=\|\wh f\|_2\lesssim R \|\wh f\|_\infty$, so
\begin{equation}\label{est2}
    \|f_{O^{(s)}}\|_2^{2}\lesssim D^{-2s_c}R^{2} \|\wh f\|_\infty^2.
\end{equation}

Plugging \eqref{est1} and \eqref{est2} into \eqref{plugin}, we obtain
\begin{equation}\label{plugin2}
\begin{split}
    \|\br Ef\|_{L^{13/4}(B_R)}^{13/4}\lesssim R^{\e/2} D^{-s_c/2}R^{5/4}\|f\|_2^{2}\|\wh f\|_\infty^{5/4}\le R^{\e/2} R^{5/4}\|f\|_2^{2}\|\wh f\|_\infty^{5/4},
\end{split}
\end{equation}
which is \eqref{ineqbroad}.

\smallskip

Next, we consider the case that STATE$(s)$=tang. Similar to \eqref{plugin}, we have
\begin{equation}
    \|\br Ef\|_{L^{13/4}(B_R)}^{13/4}\lesssim R^{O(\e^2)}\sum_{O^{(s)}}  \| Ef_{O^{(s)}}\|_{L^{13/4}(O^{(s)})}^{13/4}.
\end{equation}
By \eqref{ittang3}, we have
\begin{equation}\label{biltang}
\begin{split}
    \|\br Ef\|_{L^{13/4}(B_R)}^{13/4} &\lesssim R^{O(\e)}r_s^{1/16}
    \sum_{O^{(s)}} \|f_{O^{(s)}}\|_2^{13/4}.
\end{split}
\end{equation}
We also have \eqref{est1} and \eqref{est2}. But here in the tangent case, we need a new estimate of $\|f_{O^{(s)}}\|_2$ other than \eqref{est2}. We will obtain it using polynomial Wolff. Fix a cell $O^{(s)}$. Let us first recall how we obtain $f_{O^{(s)}}$ from $f$. We have the chain that indicates the relations between cells at different scales:
$$ O_s<O_{s-1}<\cdots<O_1<B_R(=O_0). $$
For $0\le u\le s$, let $B_{r_u}$ be the ball of radius $r_u$ that contains $O_u$. $f_{O^{(u)}}$ is obtained as follows. We do wave packet decomposition for $f_{O^{(u-1)}}$ in $B_{r_{u-1}}$  and choose a subset of wave packets of $f_{O^{(u-1)}}$ that are related to the cell $O^{(u)}$. Then we define $f_{O^{(u)}}$ to be the sum of these chosen wave packets. We may denote by $\T_{O^{(u)}}$ the set of these wave packets and write the relation as
\begin{equation}\label{frelation}
    f_{O^{(u)}}=\sum_{T\in\T_{O^{(u)}}} (f_{O^{(u-1)}})_T.
\end{equation} 

To compare wave packets at different scale, we need the following definition.
\begin{definition}\label{tuberelation}
For two tubes $T=T_{\theta,v}\in \T[B_{r}(x)]$ and $T'=T'_{\theta',v'}\in \T[B_{r'}(x')]$
with $r'\le r$. We say $T'<T$ if 
\begin{equation}
\mathrm{dist}(\theta,\theta') \lesssim r'^{-1/2} \;\; \text{ and }  \;\;
\mathrm{dist}(T ,T') \lesssim r^{1/2+\delta}.
\end{equation}
\end{definition}

Now we define the tube set 
\begin{equation}\label{relatedtube}
    \T_{O^{(u-1)},O^{(u)}}:=\{T\in \T_{O^{(u-1)}}: T'<T\textup{~for~some~}T'\in\T_{O^{(u)}}\},
\end{equation} 
and the function
$$ f_{O^{(u-1)},O^{(u)}}:= \sum_{T\in\T_{O^{(u-1)},O^{(u)}}} (f_{O^{(u-1)}})_T. $$
Since $f_{O^{(u-1)}}$ is concentrated in wave packets from $\T_{O^{(u-1)}}$, we have
$$ f_{O^{(u-1)}}=\sum_{T\in\T_{O^{(u-1)}}} (f_{O^{(u-1)}})_T+\rap(R)\|f\|_2. $$

Plugging into \eqref{frelation} and using Lemma \ref{comparing}, we have
\begin{equation}
    f_{O^{(u)}}=\sum_{T\in\T_{O^{(u)}}} (f_{O^{(u-1)},O^{(u)}})_T+\rap(R)\|f\|_2.
\end{equation}

Motivated by the definition \eqref{relatedtube},  we now define
\begin{equation}\label{relatedtube2}
    \T_{O^{(s)}}^{\sharp}:=\{T\in \T_{B_R}: T'<T\textup{~for~some~}T'\in\T_{O^{(s)}}\},
\end{equation} 
and set
\begin{equation}\label{werewrite}
    f_{ {O^{(s)}}}^{\sharp} := \sum_{T\in  \T_{O^{(s)}}^{\sharp} } f_T.
\end{equation} 
Then we have
\begin{equation}\label{frelation2}
    f_{O^{(s)}}=\sum_{T\in\T_{O^{(u)}}} (f_{{ O^{(s)}}}^{\sharp}  )_T+\rap(R)\|f\|_2.
\end{equation}

Before proving an upper bound for $\|f_{O^{(s)}}\|_2$, we state the polynomial Wolff estimate in $\ZR^3$.
For its proof, we refer to \cite{Wu} (or Lemma 11.1 in \cite{gan2021new}).

\begin{lemma}[Three dimensional polynomial Wolff]\label{3dpolywolff}
Fix $R>r>R^{\e/10}>1$. Set $\de=\e^2$. Let $S$ be a two-dimensional algebraic variety of complexity at most $E=O(R^{\e^6})$. Let $B_r$ be a ball contained in $B_{10R}$.
We define 
\begin{equation}
\wt S:= B_{10R}\cap \bigcup_{\begin{subarray}{c}
     T\textup{~a~}r\times r^{1/2+\de}\textup{~tube}  \\
     T\subset N_{2r^{1/2+\de}}(S\cap B_r) 
\end{subarray} } \textup{Fat}_{\frac{R}{r}}(T).
\end{equation}
Here, $\textup{Fat}_{\frac{R}{r}}T=\frac{R}{r}T$ is the dilation of $T$.
We also define
$$ \wt S':=
\bigcup_{\begin{subarray}{c}
     T\textup{~a~}r\times r^{1/2+\de}\textup{~tube}  \\
     \angle(T,e_3)< 1/10\\
     T\subset \wt S 
\end{subarray} }T. 
$$
Then for every $a \in [-2R,2R]$, we have
\begin{equation}\label{3destmeasure}
    |\wt S' \cap (\R^{2} \times \{a\} ) |\le C(\e)R^{\e}R^{2} r^{-1/2}.
\end{equation}
\end{lemma}

We are ready to prove the following estimate
\begin{lemma}\label{L2tube}
For any $O^{(s)}\in\co^{(s)}$, we have 
\begin{equation}\label{est3}
   \|f_{O^{(s)}}\|_2^2\lesssim R^{\e^2}R^2 r_s^{-1/2}\|\wh f\|_\infty^2
\end{equation}
\end{lemma}

\begin{proof}
From \eqref{frelation2}, we have by $L^2$-orthogonality that 
$ \|f_{O^{(s)}}\|_2\lesssim \|f_{ O^{(s)}}^{\sharp}\|_2. $
Set 
$$ X:=\big(\bigcup_{ T\in \T^\sharp_{O^{(s)}}} 10T\big) \cap \R^2\times\{0\}.  $$
We claim that 
\begin{equation}\label{weclaim}
     \|f^\sharp_{O^{(s)}}\|_2\lesssim\|1_X \wh f \|_2.
\end{equation}
If it holds, then
we use polynomial Wolff estimate that $|X|\lesssim R^{\e^2} R^2 r_s^{-1/2}$, and together with H\"older's inequality, so that we obtain
\begin{equation}
    \|f_{O^{(s)}}\|_2^2\lesssim\|1_X \wh f \|_2^2\lesssim |X|\|\wh f\|_\infty^2\lesssim R^{\e^2}R^2 r_s^{-1/2}\|\wh f\|_\infty^2.
\end{equation}

Now, we prove \eqref{weclaim}. From the notation \eqref{recallwpd}, we can rewrite \eqref{werewrite} as
\begin{equation}
    f_{ {O^{(s)}}}^{\sharp} = \sum_{T_{\theta,v}\in  \T_{O^{(s)}}^{\sharp} }\td\psi_\theta\Big(\eta_v^\vee*\big(\psi_\theta f \big)\Big).
\end{equation} 
So,
\begin{equation}
\begin{split}
    \wh {f_{ {O^{(s)}}}^{\sharp}} &= \sum_{T_{\theta,v}\in  \T_{O^{(s)}}^{\sharp} } \wh{\td\psi_\theta}*\Big(\eta_v\big(\wh\psi_\theta*\wh f \big)\Big)\\
    &=\sum_{T_{\theta,v}\in  \T_{O^{(s)}}^{\sharp} } \wh{\td\psi_\theta}*\Big(\eta_v\big(\wh\psi_\theta* (\wh f 1_X) \big)\Big)+\sum_{T_{\theta,v}\in  \T_{O^{(s)}}^{\sharp} } \wh{\td\psi_\theta}*\Big(\eta_v\big(\wh\psi_\theta* (\wh f 1_{X^c}) \big)\Big).
\end{split}
\end{equation} 
We show that the second term on the right hand side is negligible. Note that $\eta_v$ is supported in $B_{r^{\frac{1+\de}{2}}}(v)$, and $\wh\psi_\theta$ is essentially supported in $B_{r^{1/2}}(0)$, so $\eta_v\big(\wh\psi_\theta* (\wh f 1_{X^c}) \big)$ is essentially supported in $B_{r^{\frac{1+\de}{2}}}(v)\cap (B_{r^{1/2}}(0)+X^c)$.
By the definition of $X$, we see that if $T_{\theta,v}\in  \T_{O^{(s)}}^{\sharp}$, then $10B_{r^{\frac{1+\de}{2}}}(v)\cap X^c=\emptyset$. As a result, we proved that $\eta_v\big(\wh\psi_\theta* (\wh f 1_{X^c}) \big)$ is negligible.
Ignoring the negligible term, we have
\begin{equation}
    \|f_{ {O^{(s)}}}^{\sharp}\|_2^2=\|\wh {f_{ {O^{(s)}}}^{\sharp}}\|_2^2=\bigg\|\sum_{T_{\theta,v}\in  \T_{O^{(s)}}^{\sharp} } \wh{\td\psi_\theta}*\Big(\eta_v\big(\wh\psi_\theta* (\wh f 1_X) \big)\Big)\bigg\|_2^2\lesssim \|\wh f 1_X\|_2^2,
\end{equation}
which finishes the proof.
\end{proof}

Let us return to \eqref{biltang}. We have
\begin{equation}\label{last}
\begin{split}
    \|\Br_\al Ef\|_{L^{13/4}(B_R)}^{13/4} &\lesssim R^{O(\e)}r_s^{1/16} \sum_{ O^{(s)}}\|f_{O^{(s)}}\|_2^{2} \max_{O^{(s)}}\|f_{O^{(s)}}\|_2^{5/4}\\
    (\textup{by~}\eqref{est1})&\lesssim R^{O(\e)}r_s^{1/16} D^{s_c}\|f\|_2^2 \max_{O^{(s)}}\|f_{O^{(s)}}\|_2^{5/4}
    \end{split}
\end{equation}
To estimate $\max_{O^{(s)}}\|f_{O^{(s)}}\|_2^{5/4}$, we combine $1/2$ power of \eqref{est2} and $1/8$ power of \eqref{est3} to obtain
$$ \max_{O^{(s)}}\|f_{O^{(s)}}\|_2^{5/4}\lesssim D^{-s_c}R^{5/4}r_s^{-1/16}\|\wh f\|_\infty^2. $$
Plugging into \eqref{last}, we obtain \eqref{ineqbroad}.

\section{A proof for the case \texorpdfstring{$\alpha<1$}{Lg}}

In this section, we prove Theorem \ref{mainthm} for the case: $\nabla_{\xi\xi}h$ has eigenvalues of different signs and $n=3$. The strategy is very similar to that for $\alpha=2$. The main difference is, for the case $\alpha<1$, the surface $\{(\xi,h(\xi)):\xi\in[-1,1]^2 \}$ has negative Gaussian curvature. When it comes to the restriction problem, the obstacles from this difference is well understood in \cite{guo2020restriction}. We will simply combine the ideas therein with the ideas in Section \ref{sec3}. Since the main idea is already addressed in Section \ref{sec3}, we give only a sketch of the proof here.

We first make some reductions.
As in the previous section, by discarding the contribution outside $B_R$, we may assume that the wave packets of $f$ is contained in $B_{10R}$. Moreover, it suffices to prove \eqref{mainest} for polynomials $h$ of the form
\begin{equation}\label{0519e1.8}
h(\xi_1,\xi_2):=\xi_1\xi_2+a_{2,0}\xi_1^2+a_{2,2}\xi_2^2+
\sum_{i=3}^{d}\sum_{j=0}^{i} a_{i,j}\xi_1^{i-j}\xi_2^{j}
\end{equation}
satisfying the condition
\begin{equation}\label{maininequality}
|a_{2,0}|+|a_{2,2}|+100^d\sum_{i=3}^{d}\sum_{j=0}^{i}|a_{i,j}|\le \epsilon_0.
\end{equation}
Here, $\epsilon_0$ is a small number.
Once we prove \eqref{mainest} for such polynomials, the general case follows by a simple application of Taylor's theorem 
(see Section 2 of \cite{guo2020restriction} for the details). Hence, we may assume that $h$ satisfies the condition \eqref{maininequality} and aim to prove
\begin{equation}\label{0817.43}
    \|Ef\|_{L^{13/4}(B_R)} \leq C_{\e} R^{\e}R^{(1-\frac{8}{13})}\|\widehat{f}\|_{L^{13/4}}.
\end{equation}

\subsection{Broad-narrow reduction}
We do broad-narrow reduction in this subsection. 
Since the surface has negative curvature,
the broad-narrow reduction is more involved than that for the surface with positive curvature as we did in Section \ref{sec3}. We will follow the idea from \cite{guo2020restriction}.

\subsubsection{Bad lines}

 Let $c_1(K_L^{-1})$ be a small number given by 
\begin{equation}
    c_1(K_L^{-1}):=10^{-10d}\epsilon_{0}K_L^{-1}.
\end{equation}
Take a collection $\mathbb{D}$ of points on the unit circle $S^1$ such that $\mathbb{D}$ is a maximal $c_1(K_L^{-1})$-separated set. For every $a\in c_1(K_L^{-1})\mathbb{Z} \cap [-10,10]$ and $v=(v_1,v_2) \in \mathbb{D}$ with $|v_2| \leq |v_1|$, let $l_{1,a,v}$ denote the line passing through the point $(0,a)$ with direction vector $v$. Similarly, for every $a \in c_1(K_L^{-1})\mathbb{Z} \cap [-10,10]$ and $v=(v_1,v_2) \in \mathbb{D}$ with $|v_2|>|v_1|$, let $l_{2,a,v}$ denote the line passing through $(a,0)$ with direction vector $v$.

Let us now define bad lines.
Suppose that a line $l=l_{1,a,v}$ is given. We define an affine transformation $M_l$ associated to the line $l$ in the following way. First, let $M_{l,1}$ be the action of translation that sends $(0,a)$ to the origin and let $M_{l,2}$ be the rotation mapping $v$ to the point $(1,0)$. Define
\begin{equation}
    M_l:=M_{l,2} \circ M_{l,1}.
\end{equation}
We write the polynomial $(h \circ M_l^{-1})(\xi_1,\xi_2)$ as 
\begin{equation}\label{0421e2.2}
(h \circ M_l^{-1})(\xi_1,\xi_2)=\sum_{i=0}^d\sum_{j=0}^{i}c_{i,j}\xi_1^{i-j}\xi_2^j.
\end{equation}
By the assumption \eqref{maininequality}, we see that $|c_{2,1}| \geq 100^{-1}$. 
The line $l_{1,a,v}$ is called \textit{a bad line} provided that
\begin{equation}\label{0421e2.3}
\max_{i=2,\ldots,d}(|c_{i,0}|) \leq 10^{-5d}\epsilon_0K_L^{-1}=10^{5d} c_1(K_L^{-1}).
\end{equation}
We define a bad line for $l_{2,a,v}$ in a similar way, with the role of $\xi_1$ and $\xi_2$ in \eqref{0421e2.2} and \eqref{0421e2.3} exchanged.
Take $\iota\in \{1, 2\}$. Let $L_{\iota,a,v}$ denote the $c_1(K_L^{-1})$-neighborhood of $l_{\iota,a,v}$ and call it \textit{a bad strip}.
We denote {the} collection of all the bad strips by
\begin{equation}
\mathbb{L}:=\{ 
L_{\iota,a,v }: l_{\iota,a,v} \text{ is a bad line}; \iota=1, 2
\}.
\end{equation}

\subsubsection{Broad function}
Consider dyadic numbers
\begin{equation}
    K_L=K_{d+1} \ll K_d \ll \cdots \ll  K_1 \ll K_0=K.
\end{equation}
Let $M:=10^{20d}\epsilon_0^{-1}$. Define $\mathcal{P}(K^{-1},A)$ to be a collection of all dyadic squares with side length $K^{-1}$ in a set $A$. We sometimes use $\mathcal{P}(K^{-1})$ for $\mathcal{P}(K^{-1},[-1,1]^2)$.

For every  $\alpha=(\alpha_0,\ldots,\alpha_{d+1}) \in (0,1)^{d+2}$, we {say that $x\in \R^3$ is} $\alpha$-broad if 
\begin{equation}\label{0511e3.2}
    \max_{L_1,\ldots,L_{M} \in \mathbb{L} } \Big|\sum_{\substack{ \tau \in \mathcal{P}(K^{-1},\, \cap_{i=1}^{M} L_i) } } Ef_{\tau}(x)\Big| \leq \alpha_{d+1} |Ef(x)|
\end{equation}
and
\begin{equation}\label{0507e3.3}
\begin{split}
    &\max_{\tau_j \in \mathcal{P}(K_j^{-1})} |Ef_{\tau_j}(x)|
    \\&+
    \max_{\upsilon_j \in \mathcal{P}(K_j^{-1})}
        \max_{L_1,\ldots,L_{M} \in \mathbb{L} } \Big|\sum_{\substack{ \tau \in \mathcal{P}(K^{-1},\,
        \upsilon_j \cap (\cap_{i=1}^{M} L_i ))}} Ef_{\tau}(x)\Big|
    \leq \alpha_j |Ef(x)|
\end{split}
\end{equation}
for every $j=0,\ldots,d+1$. 
For $\alpha \in (0,1)^{d+2}$ and $r \in \R$, we use the notation $r\alpha=(r\alpha_0,\ldots,r\alpha_{d+1})$.
We let $\br Ef(x)$ denote the function which is $|Ef(x)|$ if $x$ is an $\alpha$-broad point of $Ef$ and 0 otherwise. 

We will prove the following estimate.

\begin{theorem}\label{broadfunctionestimate}
 For every $\epsilon>0$ there exist dyadic numbers $K,K_1,\ldots,K_{d+1}$ with 
\begin{equation}
    K_{d+1}(\epsilon) \ll K_d(\epsilon) \ll \cdots \ll  K_1(\epsilon) \ll K_0(\epsilon)=K
\end{equation} 
and $\alpha_j \sim K_j^{-\epsilon}$ such that for every $R \geq 1$, we have 
\begin{equation}\label{broadfunctionestimate'}
\begin{split}
	\|\mathrm{Br}_{\alpha}Ef\|_{L^{13/4}(B_R)}
	\leq C_{\epsilon,d}R^{\epsilon}R^{2(\frac12-\frac{4}{13})}\|f\|_{L^2}^{8/13}
	\|\widehat{f}\|_{L^{\infty}}^{5/13}.
	\end{split}
	\end{equation}
Moreover,  $\lim_{\epsilon \rightarrow 0}K_{d+1}(\epsilon) \rightarrow +\infty$.
\end{theorem}

\subsection{Proof of Theorem \ref{mainthm} for the negative curvature case}

In this subsection, we prove \eqref{0817.43} by assuming Theorem \ref{broadfunctionestimate}. By interpolation, it suffices to prove
\begin{equation}\label{0817.414}
    \|Ef\|_{L^{13/4}(B_R)} \leq C_{\e,d} R^{\e}R^{2(\frac12-\frac{4}{13})}\|{f}\|_{L^{2}}^{8/13}\|\widehat{f}\|_{L^{\infty}}^{5/13}.
\end{equation}
We use induction on the radius $R$. Assume that \eqref{0817.414} is true for radius $\leq R/2$. By isotropic and anisotropic rescaling, we can prove the following lemmas.  

\begin{lemma}[cf. Lemma 3.2 of \cite{guo2020restriction}]\label{squarerescaling}
Suppose that $1 \lesssim K \leq R' \leq R/2$.
 Under the induction hypothesis, for every  square $\tau \in \mathcal{P}(K^{-1})$, we have
\begin{equation}
\|Ef_{\tau}\|_{L^{13/4}(B_{R'})} \leq
K^{-2+6\cdot\frac{4}{13}}C_{\epsilon,d}(R')^{\epsilon}
(R')^{2(\frac12-\frac{4}{13})}
\|f_{\tau}\|_{L^2}^{8/13}
\|\widehat{f}\|_{L^\infty}^{5/13}.
\end{equation}
\end{lemma}

The proof of the above lemma is essentially the same as that for Lemma \ref{rescale}. We leave out the details.

\begin{lemma}[cf. Lemma 3.3 of \cite{guo2020restriction}]\label{striprescaling} Suppose that $1 \lesssim K_{d+1} \leq R' \leq R/2$. Under the induction hypothesis,
for every $L \in \mathbb{L}$, we have
\begin{equation}
\|{Ef_{L} }\|_{L^{13/4}(B_{
R'})}
\leq (K_{d+1})^{-\frac{1}{13}}
C_{\epsilon,d}
({R'})^{\epsilon}(R')^{2(\frac12-\frac{4}{13})}\|f_L\|_{L^2}^{8/13}\|\widehat{f}\|_{L^\infty}^{5/13}.
\end{equation}
\end{lemma}

\begin{proof}[Sketch of the proof of Lemma \ref{striprescaling}] The proof is essentially given in Lemma 3.3 of \cite{guo2020restriction}, so we only give a sketch here. For simplicity, assume that our strip $L$ is $[0,1] \times [0,K_{d+1}^{-1}]$. We apply the rescaling $(\xi_1,\xi_2) \mapsto (\xi_1,K_{d+1}\xi_2)$. After the rescaling, $Ef_L$ becomes
\begin{equation}\label{417}
    K_{d+1}^{-1}\int_{[-1,1]^2}
    f_L(\xi_1,K_{d+1}^{-1}\xi_2)e((x_1,K_{d+1}^{-1}x_2,K_{d+1}^{-1}x_3) \cdot (\xi_1,\xi_2,\tilde{h}(\xi))  )\,d\xi_1 d\xi_2
\end{equation}
where $\tilde{h}$ is some function satisfying the conditions \eqref{0519e1.8} and \eqref{maininequality}. Define the function $g(\xi_1,\xi_2)=f_L(\xi_1,K_{d+1}^{-1}\xi_2)$ and denote \eqref{417} by $\tilde{E}g(x_1,K_{d+1}^{-1}x_2,K_{d+1}^{-1}x_3)$. We apply the change of variables on the physical variables $x$, and obtain
\begin{equation}\label{418}
    \|Ef_L\|_{L^{13/4}([0,R']^3)} \lesssim (K_{d+1})^{-\frac{5}{13}}\|\tilde{E}g\|_{L^{13/4}([0,R'] \times [0,R'/K] \times [0,R'/K]) }.
\end{equation}
We decompose the rectangular box into smaller squares of sidelength $R'/K$, and apply the induction hypothesis on $R$. Then \eqref{418} is bounded by
\begin{equation}
    (K_{d+1})^{-\frac{5}{13}}(\frac{R'}{K_{d+1}})^{\epsilon}(\frac{R'}{K_{d+1}})^{\frac{5}{13}}\|g\|_{2}^{\frac{8}{13}}\|\widehat{g}\|_{\infty}^{\frac{5}{13}}.
\end{equation}
By changing back to the original variables, the above term becomes
\begin{equation}
    (K_{d+1})^{-\frac{1}{13}-\epsilon}({R'})^{\epsilon}({R'})^{\frac{5}{13}}\|f\|_{2}^{\frac{8}{13}}\|\widehat{f}\|_{\infty}^{\frac{5}{13}}.
\end{equation}
This completes the proof.
\end{proof}

Let us continue the proof of \eqref{0817.414}.
 By the definition of {the} broad function, we obtain
\begin{equation}\label{0507e3.7}
\begin{split}
|Ef(x)| &\leq |\br{Ef}(x)|
+\sum_{j=0}^{d+1}
\alpha_j^{-1} \Bigg(
 \max_{\tau_j \in \mathcal{P} (K_j^{-1})}|Ef_{\tau_j}(x)|\\&
        +\max_{\upsilon_j \in \mathcal{P}(K_j^{-1})}
        \max_{L_1,\ldots,L_{M} \in \mathbb{L} } \Big|\sum_{\tau \in \mathcal{P}(K^{-1},\, \upsilon_j \cap(\cap_{i=1}^{M} L_i)) } Ef_{\tau}(x)\Big|\Bigg)
        \\&
        +
        (\alpha_{d+1})^{-1}
        \max_{L_1,\ldots,L_{M} \in \mathbb{L} } \Big|\sum_{\tau \in \mathcal{P}(K^{-1},\, \cap_{i=1}^{M} L_i) } Ef_{\tau}(x)\Big|.
\end{split}
\end{equation}
We raise both sides to the $13/4$-th power, integrate over $B_R$, replace the max by $l^{13/4}$-norms, and obtain
\begin{equation}\label{0818.418}
\begin{split}
    \int_{B_R}|Ef|^{13/4}
    & \leq C\int_{B_R}|\br{Ef}|^{13/4}
    +C\sum_{j=0}^{d+1}
\alpha_j^{-13/4} \Bigg(
 \sum_{\tau_j \in \mathcal{P} (K_j^{-1})}\int_{B_R}|Ef_{\tau_j}|^{13/4}
 \\&
        +\sum_{\upsilon_j \in \mathcal{P}(K_j^{-1})}
        \sum_{L_1,\ldots,L_{M} \in \mathbb{L} }\int_{B_R} \bigg|\sum_{\tau \in \mathcal{P}(K^{-1},\, \upsilon_j \cap(\cap_{i=1}^{M} L_i)) } Ef_{\tau}\bigg|^{13/4}\Bigg)
        \\&
        +C
        (\alpha_{d+1})^{-13/4}
        \sum_{L_1,\ldots,L_{M} \in \mathbb{L} }\int_{B_R} \bigg|\sum_{\tau \in \mathcal{P}(K^{-1},\, \cap_{i=1}^{M} L_i) } Ef_{\tau}\bigg|^{13/4}.
\end{split}
\end{equation}
We apply Theorem \ref{broadfunctionestimate} to the first term, apply H\"{o}lder's inequality, and obtain 
\begin{equation}\label{3.10}
\begin{split}
\int_{B_R}|\mathrm{Br}_{\alpha}Ef|^{13/4}
&\leq CC_{\epsilon,d}^{13/4}R^{\epsilon/100}R^{2(\frac12-\frac{4}{13})\frac{13}{4}}
\|f\|_{L^2}^{2}\|\widehat{f}\|_{L^{\infty}}^{\frac{5}{4}}
\\&
\leq 10^{-13d/4}C_{\epsilon,d}^{13/4}R^{13\epsilon/4}R^{2(\frac12-\frac{4}{13})\frac{13}{4}}
\|f\|_{L^2}^{2}\|\widehat{f}\|_{\infty}^{\frac{5}{4}},
\end{split}
\end{equation}
provided that $R$ is large enough so that
\begin{equation}
    10^{13d/4}R^{\epsilon/100} \leq R^{13\epsilon/4}.
\end{equation}
This takes care of the contribution from the first term. 
Let us bound the second term on the right hand side of \eqref{0818.418}. We first break our ball $B_R$ into smaller balls $B_{R/2}$, and apply Lemma \ref{squarerescaling} with $R'=R/2$, and  obtain
\begin{equation}
\begin{split}
\sum_{\tau_j \in \mathcal{P} (K_j^{-1})} \int_{B_R}|Ef_{\tau_j}|^{13/4} &\leq
K_j^{-100\epsilon}
C_{\epsilon,d}^{13/4}R^{13\epsilon/4}
R^{2(\frac12-\frac{4}{13})\frac{13}{4}}
\sum_{\tau_j \in \mathcal{P} (K_j^{-1})}
\|f_{\tau_j}\|_{L^2}^{2}\|\widehat{f}\|_{L^{\infty}}^{5/4}
\\&
\leq
K_j^{-100\epsilon}
C_{\epsilon,d}^{13/4}R^{13\epsilon/4}
R^{2(\frac12-\frac{4}{13})\frac{13}{4}}
\|f\|_{L^2}^{2}\|\widehat{f}\|_{L^{\infty}}^{5/4}.
\end{split}
\end{equation}
Recall that $\alpha_j \sim K^{-\epsilon}$.
Hence, the second term is also harmless. The third term can be dealt with by following the same argument. We leave out the details.
Let us move on to the fourth term. By Lemma \ref{striprescaling}, we  obtain
\begin{equation}
\begin{split}
&
\sum_{L_1,\ldots,L_{M} \in \mathbb{L} }
\int_{B_R}\bigg| \sum_{\tau \in \mathcal{P}(K^{-1},\, \cap_{i=1}^{M} L_i) }Ef_{\tau}
\bigg|^{13/4} 
\\&
\leq
(K_{d+1})^{-100\epsilon}
C_{\epsilon,d}^{13/4}
R^{13\epsilon/4}
R^{2(\frac12-\frac{4}{13})\frac{13}{4}}
\Big(
\sum_{L_1,\ldots,L_{M} \in \mathbb{L} }
\|f_{\cap_{i=1}^{M} L_{i}}\|_{L^2}^{2} \Big)\|\widehat{f}\|_{L^{\infty}}^{5/4}.
\end{split}
\end{equation}
By Lemma 3.1 of \cite{guo2020restriction}, we have
\begin{equation}
    \sum_{L_1,\ldots,L_{M} \in \mathbb{L} }
\|f_{\cap_{i=1}^{M} L_{i}}\|_{L^2}^{2} \leq C M^M\|f\|_2^2.
\end{equation}
Notice that the term $M^M$ is harmless. Also, recall that $\alpha_{d+1} \sim K_{d+1}^{-\epsilon}$. Combining all the information, we obtain the desired bound for the fourth term.
This completes the proof of \eqref{0817.414}.

Now it suffices to prove Theorem \ref{broadfunctionestimate}, which we will discuss in the next section.

\subsection{Proof of Theorem \ref{broadfunctionestimate}}

In this subsection, we prove Theorem \ref{broadfunctionestimate}.
Fix $\e< 1/100$.
Take the parameters $
    D=R^{\e^6}$ and $ \delta=\e^2$. Let $K$ be a large number depending on $\e$ but independent of $R$. We write $f=\sum_{\tau}f_{\tau}$. Here, $\tau$ is a ball of radius $K^{-1/2}$. Then we apply the wave packet decomposition to $f_{\tau}$: $f_{\tau}=\sum_T f_{\tau,T}$. We sometimes write $f_T$ for $f_{\tau,T}$. 
We follow the proof in Section \ref{sec32}, and obtain an inequality similar to \eqref{cellwall}:
\begin{equation}
    \|\br Ef\|_{L^{13/4}(B_R)}^{13/4} \lesssim \sum_{O_i^{(1)}  }\|\mathrm{Br}_{\alpha}Ef\|_{L^{13/4}(O_i^{(1)}) }^{13/4}
    +\|\mathrm{Br}_{\alpha}Ef_{}\|_{L^{13/4}(W) }^{13/4}.
\end{equation}
Define $f_i$ by
\begin{equation}
    f_{i}:=\sum_{\tau}\sum_{T \in \T_i}f_{\tau,T}.
\end{equation}
By Lemma 4.2 of \cite{guo2020restriction}, for every $x \in O_i^{(1)}$, we have
\begin{equation}
\mathrm{Br}_{\alpha}Ef(x)
\leq \mathrm{Br}_{2\alpha}Ef_i(x) + O(R^{-900}\|f\|_{L^2}).
\end{equation}
Cover the wall $W$ by balls $B_k$ of radius $R^{1-\delta}$. 
Define $\T_{k,-}$ and $\T_{k,+}$ as in Definition \ref{deftang} and \ref{deftrans}. Define the functions $f_{\tau,k,+}, f_{k,+}$ and $f_{I,k,+}$ as in \eqref{313} and \eqref{314}.
Lastly, we define the bilinear function. Compared with \eqref{defbil}, the bilinear function is more involved here.
\begin{equation}
\mathrm{Bil}(Ef):=\sum_{\substack{(\tau,\tau')\in \mathcal{P}(K^{-1}) \times \mathcal{P}(K^{-1}):  \\
(\tau,\tau') \, \text{is a good pair}}} |Ef_{\tau}|^{1/2}|Ef_{\tau'}|^{1/2}.
\end{equation}
In the above formula, we use the notation ``good pair" for which we didn't give the definition in this paper.
In fact, the definition is not important here because we will simply cite the following two lemmas as a black box. We refer to Section 3.2 of \cite{guo2020restriction} for the definition of the ``good pair" and more discussions.

\begin{lemma}[Lemma 5.4 of \cite{guo2020restriction}]\label{bourgainguth}
Suppose that $\alpha \in (0,1)^{d+2}$ satisfies 
\begin{equation}
    K_j^{-\epsilon} \leq \alpha_j \leq K_{j+1}^{-100d},\;\;\;\; K_{d+1}^{-\epsilon} \leq \alpha_{d+1} \leq 10^{-100d}
\end{equation}
for every $j=0,\ldots,d$.
Then for every $x \in B_k \cap W$
\begin{equation}\label{0807.6.17}
\begin{split}
\mathrm{Br}_{\alpha}Ef(x) &\leq 2\sum_I \mathrm{Br}_{200d^2\alpha}Ef_{I,k,+}(x)
\\&+K_0^{100}\mathrm{Bil}(Ef_{k,-})(x) + R^{-50}\|f\|_2.
\end{split}
\end{equation}
\end{lemma}
The summation $\sum_I$ runs over all possible collections of squares with sidelength $K_0^{-1}$. This summation does not play a significant role.

\begin{lemma}[(5.127) of \cite{guo2020restriction}]\label{0807.6.2}
\begin{equation}\label{431}
\begin{split}
\|\mathrm{Bil}(Ef_{k,-})\|_{L^{13/4}(B_k \cap W)}
\lesssim R^{C\delta}R^{\frac{1}{52}}\big(\sum_{\tau\in \mathcal{P}(K^{-1}) } \|f_{\tau,k,-}\|_2^2 \big)^{1/2}.
\end{split}
\end{equation}
\end{lemma}

The above lemma is proved in (5.127) of \cite{guo2020restriction}, which is a consequence of the $L^4$-estimate (Lemma 5.6 of \cite{guo2020restriction}). We refer to the paper for the proof.
\\

By Lemma \ref{bourgainguth}, we have the counterpart of \eqref{3term}:
\begin{equation}
\begin{split}
    \|\Br_\al Ef\|_{L^{13/4}(B_R)}^{13/4}&\lesssim \sum_{O_i^{(1)}}\!\|\Br_\al Ef\|^{13/4}_{L^{13/4}(O_i^{(1)})}+\sum_{k,I}\|\Br_{(200d^2)\al}Ef_{I,k,+}\|^{13/4}_{L^{13/4}(B_k\cap W)}\\
    &+K^{100}\|\bil(Ef_{j,-})\|_{L^{13/4}(W)}^{13/4}+\rap(R)\|f\|_2.
\end{split}
\end{equation}
The only difference between this inequality and \eqref{3term} is that $\mathrm{Br}_{\alpha}Ef_{I,k,+}$ is replaced by $\mathrm{Br}_{(200d^2)\alpha}Ef_{I,k,+}$. This does not make any change for the rest of the proof. Note that the bilinear function also has the nice estimate \eqref{431} as a counterpart of \eqref{327}. 

\smallskip

So far, we have finished the part of the proof corresponding to Section \ref{defbroad} and Section \ref{sec32}. For each main estimate in Section \ref{defbroad} or Section \ref{sec32}, we can find in this section its counterpart which has been adapted to the negative curvature setting. What remains is to do iteration as we did in Section \ref{sec33}.
Although we are in the negative curvature setting,  
we can follow the arguments in Section 3.3 line by line to finish the proof of Theorem \ref{broadfunctionestimate}. We do not reproduce it here.

\section{A proof for the case \texorpdfstring{$\alpha>1$}{Lg}}
In this section, we prove that when the surface $\{(\xi,h(\xi)),\xi\in[-1,1]^{n-1}\}$ has positive principle curvatures we have
\begin{equation}
\label{main-alpha>1}
    \|Ef\|_{L^p(B_R)} \leq C_{p,\e} R^{\e}R^{(n-1)(\frac12-\frac{1}{p})}\|\widehat{f}\|_{L^p}
\end{equation}
for every $p \geq p_n$. 
Using Bourgain-Guth's broad-narrow reduction in \cite{Bourgain-Guth}, \eqref{main-alpha>1} boils down to the following $k$-broad estimate
\begin{equation}
\label{k-broad}
    \|Ef\|_{\mathrm{BL}_{k,A}^p(B_R)} \leq C_{p,\epsilon}R^{\epsilon}
     R^{(n-1)(\frac12-\frac{1}{p})}\|{f}\|_{L^2}^{2/p}\|\widehat{f}\|_{L^{\infty}}^{1-2/p},
\end{equation}
where the $k$-broad norm is defined in \cite{guth2018} Page 86. The reduction from \eqref{main-alpha>1} to \eqref{k-broad} is quite standard, and we refer to \cite{guth2018} Section 9 for details.

The main estimate is as follows.

\begin{theorem}\label{thm61}
Let $2 \leq k \leq n-1$ and
\begin{equation}\label{defpk}
    p> p_n(k):=2+\frac{6}{2(n-1)+(k-1)\prod_{i=k}^{n-1} \frac{2i}{2i+1}}.
\end{equation}
Then for every $\epsilon>0$ and $R \geq 1$, we have
\begin{equation}\label{defpk2}
    \|Ef\|_{\broad(B_R)} \leq C_{p,\epsilon}R^{\epsilon}
     R^{(n-1)(\frac12-\frac{1}{p})}\|f\|_{L^2}^{2/p}\|\widehat{f}\|_{L^{\infty}}^{1-2/p}.
\end{equation}
\end{theorem}

Before we start the proof, let us recall some notations from \cite{hickman2020note}.  

\begin{definition}
A grain is defined to be a pair $(S,B_r)$ where $S \subset \R^n$ is a transverse complete intersection and $B_r \subset \R^n$ is a ball of radius $r>0$.
A multigrain is an $(m+1)-$tuple of grains
\begin{equation}
    \vec{S}_m=(\mathcal{G}_0,\ldots,\mathcal{G}_m), \;\;\;\; \mathcal{G}_i=(S_i,B_{r_i})
\end{equation}
satisfying 
\begin{enumerate}
    \item $\mathrm{codim}S_i=i$ for $0 \leq i \leq m$,
    \item $S_m \subset S_{m-1} \subset \ldots \subset S_0$,
    \item $B_{r_m} \subset B_{r_{m-1}} \subset \ldots \subset B_{r_0}$.
\end{enumerate}
\end{definition}

\begin{definition}\label{def62}
Let $\vec{S}_m=(\mathcal{G}_0,\ldots,\mathcal{G}_m)$ be a multigrain with
\begin{equation}
    \mathcal{G}_i=(S_i,B_{r_i}(y_i)) \;\; \text{for} \;\; 0 \leq i \leq m.
\end{equation}
Define $\T[\vec{S}_m]$ to be a sub-collection of $\T[B_{r_0}(y_0)]$ for which element satisfies the following nested tube hypothesis: there exists $T_{\theta_i,v_i} \in \T[B_{r_i}(y_i)]$ for $1 \leq i \leq m$ such that
\begin{enumerate}
    \item $\mathrm{dist}(\theta_i,\theta_j) \lesssim r_j^{-1/2}$,
    \item $\mathrm{dist}(T_{\theta_j,v_j}(y_j), T_{\theta_i,v_i}(y_i) ) \lesssim r_i^{1/2+\delta}$,
    \smallskip
    \item $T_{\theta_j,v_j}(y_j) \subset N_{r_j^{1/2+\delta_j}}S_j$
\end{enumerate}
hold for all $0 \leq i \leq j \leq m$.
\end{definition}

The proof of Theorem \ref{thm61} is based on an algorithm of polynomial partitioning. Let us state it below. 

\begin{algorithm}[\cite{hickman2020note} page 9]\label{1algorithm}\hfill
{\rm

{\bf Inputs: } Fix $R \gg 1$ and let $f$ be a smooth function satisfying
\begin{equation}
    \|Ef\|_{\broad(B_R)} \geq C_{p,\epsilon}R^{\epsilon}
     \|{f}\|_{L^2 }.
\end{equation}
Then we have the following {\bf Outputs}:

\medskip

\noindent$\bullet$ $\mathcal{O}$ a finite collection of open subsets of $\R^n$ of diameter at most $R^{\e}$.

\smallskip

\noindent$\bullet$ A codimension $m$ ($0 \leq m \leq n-k$) and an integer parameter $1 \leq A_{m+1} \leq A$.

\smallskip

\noindent$\bullet$ An $(m+1)$-tuple of:
\begin{enumerate}
    \item Scales $\vec{r}=(r_0,\ldots,r_m)$ satisfying $R=r_0>r_1>\ldots>r_m;$
    \item Large  parameters $\vec{D}=(D_1,\ldots,D_{m+1})$.
\end{enumerate}

\smallskip

\noindent$\bullet$ For $0\leq l \leq m$ a family $\vec{\mathcal{S}}_l$ of level $l$ multigrains. Each $\vec{S}_l \in \vec{\mathcal{S}}_l$ has multiscale $\vec{r}_l=(r_0,\ldots,r_l)$. 

\smallskip

\noindent$\bullet$ For $0 \leq l \leq m $ an assigmment of a function $f_{\vec{S}_l}$ to each $\vec{S}_l \in \vec{\mathcal{S}}_l$.
\smallskip

The above data is chosen so that the following properties hold:
\begin{enumerate}
\item Define $M(\vec{r},\vec{D}):=\big(\prod_{i=1}^{m}D_i \big)^{m\delta} \big(\prod_{i=1}^{m}r_i^{(\beta_{i-1}-\beta_i)/2}D_i^{(\beta_{i-1}-\beta_m)/2} \big)$. Then
\begin{equation}\label{67}
    \|Ef\|_{\broad(B_R)} \leq M(\vec{r},\vec{D})\|f\|_2^{1-\beta_m} \Big( \sum_{O \in \mathcal{O}}\|Ef_{O}\|_{\mathrm{BL}_{k,A_{m+1}}^{p_m}(O)}^{p_m} \Big)^{\beta_m/p_m}
\end{equation}
\item 
\begin{equation}\label{68}
     \sum_{O \in \mathcal{O}}\|Ef_{O}\|_2^2 \lesssim \big(\prod_{i=1}^{m+1}D_i^{1+\delta} \big)R^{\epsilon}\|f\|_2^2
\end{equation}
\item For $1 \leq l \leq m$
\begin{equation}\label{69}
     \max_{O \in \mathcal{O}}\|f_{O}\|_2^2 \lesssim r_l^{-l/2}\prod_{i=l+1}^{m+1}r_i^{-1/2}D_i^{-(n-i)+\delta}R^{\epsilon}\max_{\vec{S}_l \in \vec{\mathcal{S}}_l }\|f_{\vec{S}_l } \|_2^2,
\end{equation}
where $r_{m+1}:=1$.

\item For each multigrain $\vec{S}_l \in \vec{\mathcal{S}}_l$, we define
\begin{equation}
    f^{\sharp}_{\vec{S}_l}:=\sum_{T \in \T[\vec{S}_l] }f_{T},
\end{equation}
where $\T[\vec{S}_l]$ is defined in Definition \ref{def62}. For $1 \leq l \leq m$,
\begin{equation}\label{611}
    \|f_{\vec{S}_l } \|_2^2 \lesssim r_l^{l/2}\big(\prod_{i=1}^{l}r_i^{-1/2}D_i^{\delta} \big)R^{\epsilon}\|f^{\sharp}_{\vec{S}_l}\|_2^2.
\end{equation}
\end{enumerate}

}

\end{algorithm}

This finishes the statement of the algorithm.
We are now ready to  prove Theorem \ref{thm61}. Since $O \in \mathcal{O}$ has radius at most $R^{\epsilon}$, we have
\begin{equation}
    \|Ef_{O}\|_{\mathrm{BL}_{k,A_{m+1}}^{p_m}(O)} \lesssim R^{\epsilon}\|f_O\|_2.
\end{equation}
Combining this inequality with \eqref{67} and \eqref{68}, we have
\begin{equation}\label{613}
    \|Ef\|_{\broad(B_R)} 
    \lesssim 
    \prod_{i=1}^{m+1}r_i^{(\beta_{i-1}-\beta_i)/2}D_i^{\beta_{i-1}/2-(1/2-1/p)+O(\delta)}R^{\e}\|f\|_2^{2/p}\max_{O \in \mathcal{O}}\|f_O\|_2^{1-2/p}.
\end{equation}

In order to estimate $\max_O \|f_O\|_2$, we will apply \eqref{69} and \eqref{611}, and combine them with the following lemma. This lemma plays a role of the counterpart of Lemma 4.3 of \cite{hickman2020note}.

\begin{lemma}\label{lem54}
For $m \leq l \leq n$,
\begin{equation}
    \max_{\vec{S}_l \in \vec{\mathcal{S}}_l }\|f^{\sharp}_{\vec{S}_l}\|_2^2 \lesssim \Big( \prod_{i=1}^{l}r_i^{-1/2} \Big)R^{\e}R^{n-1}\| \widehat{f}\|_{\infty}^2.
\end{equation}
\end{lemma}

The main ingredient of the proof of Lemma \ref{lem54} is the nested polynomial Wolff.

\begin{lemma}[Nested polynomial Wolff]\label{lemzahl} 
Fix $r_0\ge r_{1}\ge \cdots\ge r_{m}>0$ and $\rho_0\ge \rho_{1}\ge \cdots\ge \rho_m>0$ so that $1\ge \frac{\rho_m}{r_m}\ge \frac{\rho_{m-1}}{r_{m-1}}\ge \cdots \ge \frac{\rho_0}{r_0}$. Let $S_0\supset S_{1}\supset \cdots \supset S_{m}$ be semi-algebraic sets of complexity at most $E$ such that for each $i$, $S_i$ is a manifold with codimension $i$ contained in $B_{r_i}(x_i)$. We recursively define another sequence of sets $\wt S_i$ $(0\le i\le m)$.
Set $\wt S_m:=N_{2\rho_m}(S_m)$. For each $i=1,\cdots, m$, define
\begin{equation}
\wt S_{i-1}:= N_{2\rho_{i-1}}(S_{i-1})\cap \bigcup_{\begin{subarray}{c}
     T\textup{~a~}\rho_i\times r_i\textup{~tube}  \\
     T\subset \wt S_{i} 
\end{subarray} } \textup{Fat}_{\frac{r_{i-1}}{r_i}}(T).
\end{equation}
We also define
$$ \wt S_0':=\bigcup_{\begin{subarray}{c}
     T\textup{~a~}\rho_0\times r_0\textup{~tube}  \\
     \angle(T,e_n)<1/10\\
     T\subset \wt S_{0} 
\end{subarray} }T $$
Then for every $a \in [-2r_0,2r_0]$, we have
\begin{equation}\label{estmeasure}
    |\wt S_0' \cap (\R^{n-1} \times \{a\} ) |\le C(n,E,\e)r_0^{\e}r_0^{n-1} \prod_{i=1}^{m}\frac{\rho_i}{r_i}.
\end{equation}
\end{lemma}

\begin{remark}
\rm 

This is a higher dimensional version of Lemma \ref{3dpolywolff}. One minor difference is that the constant $C(n,E,\epsilon)$ here behaves badly when the complexity $E$ is big, so in order to apply this lemma we have to set $E$ as an absolute constant that only depends on $\e$.
\end{remark}

Since the proof of Lemma \ref{lemzahl} is implicitly contained in \cite{zahl2021new}, we just sketch the proof here. First, by (2.33) in \cite{zahl2021new}, we have the estimate $$|\wt S_0|\le C(n,E,\e)r_0^{\e}r_0^{n} \prod_{i=1}^{m}\frac{\rho_i}{r_i}.$$
Next, we follow the proof of Lemma 2.6 in \cite{zahl2021new} where we set $S$ to be $\wt S_0'$. From (2.17) in \cite{zahl2021new} where we set $\lambda=r_0$ and $A=1$, we obtain \eqref{estmeasure}. We are ready to prove Lemma \ref{lem54}.

\begin{proof}[Proof of Lemma \ref{lem54}]
The proof is the same as that for Lemma \ref{L2tube}.
Set
    \begin{equation}
        X= \Big(\bigcup_{T \in \T[\vec{S}_l]}T \Big) \cap (\R^{n-1} \times \{0\}).
    \end{equation}
We see that $\widehat{f^{\sharp}_{\vec{S}_l}}$ is essentially supported in $X$, so
\begin{equation}\label{applyholder}
    \|\widehat{f^{\sharp}_{\vec{S}_l}}\|_2^2\lesssim \|1_X \widehat{f^{\sharp}_{\vec{S}_l}}\|_2^2 \lesssim \|1_X\widehat{f}\|_2^2.
    \end{equation}

Next, we use Lemma \ref{lemzahl} to estimate $|X|$. Set $m=l$ and $\rho_i=r_i^{1/2+\de_i}$ ($0\le i\le l$) in Lemma \ref{lemzahl} so that we obtain a set $\wt S_0'$ and the estimate \eqref{estmeasure}. By the definition of $\T[\vec{S}_l]$ in Definition \ref{def62}, we see that $$ \bigcup_{T \in \T[\vec{S}_l]}T \subset \wt S_0'. $$
As a result,
$$ |X|\le |\wt S_0 \cap (\R^{n-1} \times \{a\} ) |\le C(n,E,\e)r_0^{\e}r_0^{n-1} \prod_{i=1}^{m}\frac{\rho_i}{r_i}. $$
Noting that $r_0=R$, we apply H\"older's inequality to \eqref{applyholder} to finish the proof of Lemma \ref{lem54}.
\end{proof}

By \eqref{69}, \eqref{611}, and the above lemma, we have
\begin{equation}
    \max_{O \in \mathcal{O}}\|f_O\|_2^2
    \lesssim
    R^{\e}R^{n-1}
    \Big(\prod_{i=1}^{m}r_i^{-1/2}D_i^{\delta} \Big)\Big(\prod_{i=1}^{l}r_i^{-1/2} \Big)\Big(\prod_{i=l+1}^{m+1}D_i^{-(n-i)} \Big)\|\widehat{f}\|_{\infty}^2
\end{equation}
for all $ 0\leq l \leq m$. We take a geometric average over $l$ with weight $0 \leq \gamma_l \leq 1$. Then
\begin{equation}
    \max_{O \in \mathcal{O}}\|f_O\|_2^2
    \lesssim
    R^{\e}R^{n-1}
    \prod_{i=1}^{m+1}r_i^{-(1+\sigma_i)/2}D_i^{-(n-i)(1-\sigma_i)+O(\delta)}\|\widehat{f}\|_{\infty}^2
\end{equation}
where $\sigma_i:=\sum_{j=i}^{m}\gamma_j$ for $0 \leq i \leq m$ and $\sigma_{m+1}=0$. Applying this inequality to the right hand side of \eqref{613}, we obtain
\begin{equation}
\|Ef\|_{\broad(B_R)} 
        \lesssim R^{(n-1)(1-\frac2p)} \prod_{i=1}^{m+1}r_i^{X_i}D_i^{Y_i+O(\delta)}\|f\|_2^{2/p}\|\widehat{f}\|_{\infty}^{1-2/p}.
\end{equation}
Here $X_i$ and $Y_i$ are given on page 12 of \cite{hickman2020note}. Following the same optimization process therein gives the desired bound \eqref{defpk2}. We do not reproduce it.

\section{Appendix:  Approach using pseudo-conformal transformation}

In this section, we briefly explain how the pseudo-conformal transformation can be applied to local smoothing problems. Similar ideas date back to \cite{MR1151328}, \cite{MR2456277}. 

% \todo{Changkeun: We are identifying the words restriction and extension. How about we change the word ``extension type'' below by ``restriction type'' so that we are only using the word ``restriction''? If you agree please change it. Otherwise, please simply erase this todo and leave the manuscript.}
Define a ``restriction type" operator by
\begin{equation}\label{extension}
     E_Rf(x,t):=\int_{\R^{n-1}}e\big(\frac{R^2}{t}\big|\frac{x-ty}{R}\big|^{\frac{\alpha}{\alpha-1}}\big)\psi\big(\frac{x-ty}{R}\big)f(y)\,dy
 \end{equation}
for $f \in L^{1}([0,1]^{n-1})$. Here, $\psi$ is a  compactly supported smooth function.

\begin{conjecture}[Restriction type conjecture]\label{rescaliedrestrictionconjecture} For $p>\frac{2n}{n-1}$, $\e>0$, and $ R \geq 1$, it holds that 
\begin{equation}\label{22}
    \|E_Rf\|_{L^p([0,R]^{n-1} \times [R/2,R] )} \leq C_{p,\e} R^{\e} \|f\|_{L^p}.
\end{equation}
\end{conjecture}

Notice that when $\alpha=2$, this conjecture follows from the restriction conjecture for a paraboloid by the Taylor's theorem on the function $\psi$. In \cite{MR2456277}, it is proved that the restriction estimate for a paraboloid implies a local smoothing estimate for the Schr\"{o}dinger equation.
The goal of this section is to generalize his theorem to the general fractional Schr\"{o}dinger equations.

\begin{theorem}\label{thm22}
Let $\alpha \in (0,1) \cup (1,\infty)$.
The extension type estimate \eqref{22} for $p$ implies 
the local smoothing estimate  \eqref{12} for every $\beta > \beta_c$ for the same $p$.
\end{theorem}

The wave packets of $E_Rf$  behave very similarly to that for a restriction operator for paraboloid. It looks very plausible to prove Theorem \ref{thm12} for the case $\alpha>1$ by using the above theorem and following the arguments in \cite{GOWWZ}. On the other hand, for the case $\alpha<1$, the operator $E_Rf$ behaves differently from that for the case $\alpha>1$---the manifold associated with the operator has a negative Gaussian curvature. In particular, the operator does not seem to have a transverse equidistribution property, which is a key ingredient in \cite{GOWWZ}. Even though the restriction estimate for surfaces with negative Gaussian curvature is well understood in \cite{guo2020restriction} and \cite{buschenhenke2020fourier}, the arguments therein do not simply rely on the properties of wave packets. Thus, it is not straightforward to the authors whether their ideas can be applied to the operator $E_Rf$. We do not explore in this direction here.

\begin{proof}[Sketch of the proof of Theorem \ref{thm22}]
Recall that in Section 2 we showed that \eqref{12} follows by \eqref{225}. So it remains to prove 
\begin{equation}\label{74}
    \Big\|e^{it(-\Delta)^{\alpha/2}}f\big\|_{L^p_{x,t}(B_{R}^{n-1} \times [R/2,R])} \lesssim R^{(n-1)(\frac12-\frac{1}{p})+\e}\|f\|_{L^p(\R^{n-1})}
 \end{equation}
for $f$ whose Fourier support is contained in the annulus $\{\xi\in \R^{n-1}: |\xi| \sim 1 \}$. By plugging a smooth cutoff function $\psi$ with a suitable support ($\psi(\xi)=1$ on the annulus $\{\xi\in \R^{n-1}: |\xi| \sim 1 \}$ and $\psi$ is supported on a slightly thicker annulus),  we may write
\begin{equation}
    e^{it(-\Delta)^{\alpha/2}}f(x)=\int_{\R^{n-1}}\psi(\xi)\widehat{f}(\xi)e(x\cdot \xi +t|\xi|^{\alpha} )\,d\xi.
\end{equation}
In this way, we can write
\begin{equation}
    e^{it(-\Delta)^{\alpha/2}}f(x)=K_t*f(x),
\end{equation}
where the kernel $K_t(x)$ is defined as
\begin{equation}
    K_t(x)=\int_{\R^n}\psi(\xi)e(x \cdot \xi+t|\xi|^{\alpha})\,d\xi=e^{it(-\Delta)^{\alpha/2}}\widehat{\psi}(x).
\end{equation}
Since $\psi$ is a smooth function, it follows that $K_t(x)$ decays rapidly outside $B_R$. Hence one can assume that $f$ is supported in an $R$ ball in $\ZR^{n-1}$ without loss of generality.

Let us rewrite the kernel as
\begin{equation}
    K_t(x)=\int_{\R^{n-1}}\psi(\xi)e\big(t(\xi \cdot \frac{x}{t}+|\xi|^{\alpha})\big)\,d\xi.
\end{equation}
Note that this kernel decays rapidly for $x,t$ on the region $|x/t| \gtrsim 1$. Hence, we may assume that $|x/t| \lesssim 1$. In particular, since $t$ is restricted to the range $[R/2,R]$, one can assume $|x| \lesssim R$.
For simplicity, let us introduce the notation $\tilde{x}:=x/t$.
Define the phase function $\phi(\xi):=\xi \cdot \tilde{x}   +|\xi|^{\alpha}$. Then
\begin{equation}
    \nabla_{\xi}\phi(\xi)=\tilde{x}+\alpha |\xi|^{\alpha-2}\xi
\end{equation}
and this function vanishes only at $\xi=\xi_c:=(-\tilde{x}|\tilde{x}|^{-1})(\alpha^{-1}|\tilde{x}|)^{1/(\alpha-1)}$. 
By the method of stationary phase (see Theorem 7.7.5 of \cite{MR1996773}) and considering only the main term (the other terms can be handled similarly), one can pretend that
\begin{equation}
    K_t(x) \sim t^{-(n-1)/2}e(t\phi(\xi_c))\psi(\xi_c).
\end{equation}
Simple calculation gives
\begin{equation}
    \phi(\xi_c)=\alpha^{-\frac{1}{\alpha-1}-1}(1-\alpha)|t^{-1}{x}|^{\frac{\alpha}{\alpha-1}}.
\end{equation}
The factor $\alpha^{-\frac{1}{\alpha-1}-1}(1-\alpha)$ can be treated as 1 after using some scaling on the variable $t$. So, our operator $e^{it(-\Delta)^{\alpha/2}}f(x)$ can be morally written as
\begin{equation}\label{218}
    \begin{split}
     K_t*f(x)
     &\sim t^{-(n-1)/2}\int_{\R^n}e\big(t\big|\frac{x-y}{t}\big|^{\frac{\alpha}{\alpha-1}}\big)\psi(\frac{x-y}{t})f(y)\,dy 
     \\&
     =:t^{-(n-1)/2}Tf(x,t)
    \end{split}
\end{equation}
where $|t|\sim_\al R$.

Now let us prove \eqref{74}. Via \eqref{218}, what we need to prove becomes
\begin{equation}\label{0729.223}
    \|Tf\|_{L^p_{x,t}(B_{R}^{n-1} \times [R/2,R])} \lesssim R^{\frac{n-1}{2}}R^{(n-1)(\frac12-\frac{1}{p})+\e} \|f\|_{L^{p}(\R^{n-1})}.
\end{equation}
Define
\begin{equation}
    \tilde{T}f(x,t):=\int_{\R^n}e\big(t^{-1}|x-ty|^{\frac{\alpha}{\alpha-1}}\big)\psi(x-ty)f(y)\,dy
\end{equation}
so that
\begin{equation}
    Tf(x,t)=\tilde{T}f\big(\frac{x}{t},\frac{1}{t}\big).
\end{equation}
Employing the pseudo-conformal transformation $(x/t,1/t) \mapsto ({x},{t})$, we have
\begin{equation}\label{228}
    \|Tf\|_{L^p(B_{R}^{n-1}\times[R/2,R])} \sim R^{\frac{n+1}{p}}\|\tilde{T}f\|_{L^p([0,1]^{n-1} \times [1/(2R),1/R])}.
\end{equation}
Note that
\begin{equation}
    \tilde{T}f(R^{-1}x,R^{-2}t)=R^{n-1}E_Rg(x,t),
\end{equation}
where $g(y):=f(Ry)$ and $E_Rg$ is the extension operator \eqref{extension}.
We apply some rescaling on the physical variables: $(x,t) \mapsto (R^{-1}x,R^{-2}t)$ and on the frequency variables: $\xi \mapsto R\xi$. After this rescaling, we obtain
\begin{equation}\label{230}
    \|\tilde{T}f\|_{L^p([0,1]^{n-1} \times [1/(2R),1/R])}
    \sim
    R^{n-1-\frac{n+1}{p}}\|E_Rg\|_{L^p(B_{R}^{n-1} \times [R/2,R])}.
\end{equation}
The hypothesis \eqref{22} gives
\begin{equation}\label{231}
    \begin{split}
 \|E_Rg\|_{L^p(B_{R}^{n-1} \times [R/2,R])} \lesssim R^{\e} \|g\|_{L^p}=R^{-(n-1)/p}\|f\|_{L^p}.
    \end{split}
\end{equation}
Combining all the inequalities \eqref{228}, \eqref{230}, and \eqref{231}, we finally get
 \begin{equation}
 \|Tf\|_{L^p(B_{R}^{n-1}\times[R/2,R])} \lesssim R^{n-1}R^{-(n-1)/p}R^{\e}\|f\|_{L^p}.
 \end{equation}
This completes the proof of \eqref{0729.223}, and thus, the proof of Theorem \ref{thm22}.
\end{proof}

\bibliographystyle{alpha}
\bibliography{reference}

\end{document}